\documentclass[a4paper,10pt,article]{amsart}
\usepackage{amsmath, amscd,amssymb,amsthm,CJKnumb,verbatim,indentfirst,dsfont,mathrsfs,ipa,graphicx,textcomp,enumerate}
\usepackage[all]{xy}
\usepackage{CJK,ipa,amsmath,booktabs,longtable}
\usepackage[numbers]{natbib}
\renewcommand{\parallel}{\hbox{/\kern -2pt/}}
\begin{CJK*}{GBK}{song}
\end{CJK*}
\newtheorem{theorem}{Theorem}[section]
\newtheorem{lemma}[theorem]{Lemma}
\newtheorem{corollary}[theorem]{Corollary}
\newtheorem{proposition}[theorem]{Proposition}
\newtheorem{conjecture}[theorem]{Conjecture}

\theoremstyle{definition}
\newtheorem{definition}[theorem]{Definition}

\newtheorem{question}[theorem]{Question}
\theoremstyle{remark}
\newtheorem{remark}[theorem]{Remark}
\newtheorem{construction}[theorem]{Construction}

\numberwithin{equation}{section}

\newcommand{\Z}{\mathbb{Z}}
\newcommand{\N}{\mathbb{N}}

\title{$L^p$ Coarse Baum-Connes Conjecture and $K$-theory for $L^p$ Roe Algebras}
\begin{document}

\author{Jianguo Zhang}
\address{Shanghai Center for Mathematical Sciences, Fudan University}
\email{jgzhang15@fudan.edu.cn}

\author{Dapeng Zhou}
\address{Research Center for Operator Algebras, East China Normal University}
\email{dpzhou@math.ecnu.edu.cn}

\maketitle
\begin{abstract}In this paper, we verify the $L^p$ coarse Baum-Connes conjecture  for spaces with finite asymptotic dimension for $p\in[1,\infty)$. We also show that the $K$-theory of $L^p$ Roe algebras are independent of $p\in(1,\infty)$ for spaces with finite asymptotic dimension.
\end{abstract}
\pagestyle{plain}
\section{Introduction}

An elliptic differential operator on a closed manifold is Fredholm. The celebrated Atiyah-Singer index theorem computes the Fredholm index \cite{AtiyahSingerI} \cite{AtiyahSingerIII}. In the recent 40 years, the Atiyah-Singer index theorem has been vastly generalized to the higher index theory \cite{WillettYuBook}\cite{YuICM}. There are two most important cases. For a manifold carrying a proper cocompact  group action, the Baum-Connes assembly map  defines a higher index in the $K$-theory of the group $C^\ast$-algebra \cite{MischenkoFumenko}\cite{Kasparov88}. For an open manifold without group actions, the coarse Baum-Connes assembly  map  defines a higher index in the $K$-theory of the Roe algebra of the manifold \cite{Roe88}.

The Baum-Connes conjecture \cite{BaumConnesHigson} and the coarse Baum-Connes conjecture \cite{HigsonRoeCBC}\cite{YuCBC} give algorithms to compute the higher indices using $K$-homology. The $K$-homology is local and much more computable. In recent years, the $L^p$ version of  the Baum-Connes and the coarse Baum-Connes conjecture are studied. The motivation for using
Banach algebras  is that they are more flexible than  $C^\ast$-algebras. The traditional $C^\ast$-algebraic method \cite{Kasparov88} is very difficult for dealing with groups with property (T) (these groups admits no proper isometric actions on Hilbert spaces).  Actually a lot of interesting groups, e.g. hyperbolic group, may have property (T). Lafforgue invented the Banach $KK$-theory and verified the Baum-Connes conjecture for a large class of groups with property (T) \cite{LafforgueInven}. In \cite{YuHyperbolicLpAction}, Guoliang Yu proved that hyperbolic groups always admit proper isometric actions on $\ell^p$ spaces. In \cite{KasparovYuPBC}, Kasparov and Yu proved that the $L^p$ Baum-Connes conjecture is true for groups with a proper isometric action on $\ell^p$ space.

In \cite{LiaoYu}, Benben Liao and Guoliang Yu proved that the $K$-theory of $L^p$ group algebras are independent of $p$ for a large class of groups, e.g. hyperbolic groups. Their proof relies on Lafforgue's results on the Baum-Connes conjecture \cite{LafforgueInven} and $L^p$ property (RD) for the group.

Yeong-Chyuan Chung developed a quantitative $K$-theory for Banach algebras \cite{ChungQuantiativeBanach} and applied this theory to compute $K$-theory of $L^p$ crossed products \cite{ChungDynamical}.
Chung showed that the $L^p$ Baum-Connes conjecture for $G$ with coefficient in $C(X)$ is true if the dynamical system $G\curvearrowright X$ has Finite Dynamical Complexity, a property introduced by Guentner-Willett-Yu \cite{GuentnerWillettYuFDC} and obtained a partial answer about $p$-independence of $L^p$ crossed products.

Motivated by Liao-Yu and Chung's result, we ask the following question: Are the $K$-theory of $L^p$ Roe algebras $B^p(X)$ independent of $p$? The main theorem of the paper provides a positive answer to this question for the spaces with finite asymptotic dimension.
\begin{theorem}[see Theorem \ref{Thm:MainThmPIndpendency}]
    Let $X$ be a proper metric space. If $X$ has finite asymptotic dimension, then $K_*(B^p(X))$ does not depend on $p$ for $p\in(1,\infty)$
\end{theorem}

The proof of the theorem relies on the $L^p$ coarse Baum-Connes conjecture. The key ingredient is the Mayer-Vietoris argument. A coarse geometric Mayer-Vietoris sequence in $K$-theory  was formulated by Higson-Roe-Yu \cite{HigsonRoeYuCoarseMV}. In \cite{YuFAD}, Guoliang Yu invented the quantitative $K$-theory and a quantitative Mayer-Vietoris sequence, and he verified the coarse Baum-Connes conjecture  for spaces with finite asymptotic dimension.
The quantitative $K$-theory is a refined version of the classical operator $K$-theory. It encodes more geometric information, and it is a powerful tool to compute the $K$-theory of Roe algebras or other $C^\ast$-algebras coming from geometry.
The quantitative $K$-theory has been generalized to general geometric $C^\ast$-algebras by Oyono-Oyono and Yu \cite{OyonoYuQuantitative}\cite{OyonoYuPersistence}\cite{OyonoYuKunneth}, to Banach algebras by Yeong-Chyuan Chung \cite{ChungQuantiativeBanach}, and to groupoids by Clement Dell'Aiera \cite{DellAieraControlledGroupoid}.
It has many important applications in dynamical systems \cite{GuentnerWillettYuFDC}\cite{ChungDynamical} and coarse geometry \cite{LiWillett}\cite{ChungLiLpUniformRoe}.
 In this paper, by a similar argument of quantitative $K$-theory for $L^p$ algebras, we prove the following result.

\begin{theorem}[see Theorem \ref{Thm:LpBCCFAD}]
For any $p\in [1,\infty)$, the $L^p$ coarse  Baum-Connes conjecture holds for proper metric spaces with finite asymptotic dimension.
\end{theorem}
The result is very similar to Chung's result on Baum-Connes conjecture with coefficient for dynamical systems with finite dynamical complexity \cite{ChungDynamical}. His result is for dynamical systems or transformation groupoids, while our result is for coarse geometry or coarse groupoids.

We want to emphasize that the results in this paper do not need the condition of bounded geometry. For the similar result for spaces with bounded geometry, we could generalize the result to spaces with finite decomposition complexity, introduced by Erik Guentner, Romain Tessera and Guoliang Yu \cite{GuentnerTesseraYuRigidity}\cite{GuentnerTesseraYuGeometric}.  Our method also works for uniform $L^p$ Roe algebras. We will study the results in a separate paper.

The paper is organized in the following order: In Section 2, we recall the concept of $L^p$ Roe algebras,  $L^p$ localization algebras and $L^p$ coarse Baum-Connes conjecture. In Section 3, we study the Quantitative $K$-theory for $L^p$ algebras. In Section 4, we prove that the $L^p$ Baum-Connes conjecture is true for spaces with finite asymptotic dimension for $p\in [1,\infty)$. In Section 5, we prove that the $K$-theory of $L^p$ Roe algebras are independent of $p\in (0,1)$ for spaces with finite asymptotic dimensions. In the end, we raise some open problems for future study.

\subsection*{Acknowledgements}
    The authors would like to thank our Ph.D. advisor, Guoliang Yu, for suggesting us this project and his guidance. We are also grateful to the referee for carefully reading our manuscript and providing many valuable comments.

\section{$L^p$ Coarse Baum-Connes Conjecture}

Let $X$ be a proper metric space, $p\in [1,\infty)$. Recall that a metric space is called proper if every closed ball is compact. The proper metric space is a separable space, since compact metric space is separable.
\begin{definition}\label{Def:Lp-module}
An $L^p$-$X$-module is an $L^p$ space $E_X^p=\ell^p(Z_X)\otimes \ell^p=\ell^p(Z_X,\ell^p)$ equipped with a natural point-wise multiplication action of $C_0(X)$ by restricting to $Z_X$, where $Z_X$ is a countable dense subset in $X$, $\ell^p=\ell^p(\N)$ and $C_0(X)$ is the algebra of all complex-valued continuous functions on $X$ which vanish at infinity.
\end{definition}
We notice that this action can be extended naturally to the algebra of all bounded Borel functions on $X$.

\subsection{$L^p$ Roe algebra}

\begin{definition}
Let $E_X^p$  be an $L^p$-$X$-module and  $E_Y^p$ be an $L^p$-$Y$-module,  and let $T:E_X^p \rightarrow E_Y^p$ be a bounded linear operator. The \textit{support} of $T$, denoted $\mathrm{supp}(T)$, consists of all points $(x,y)\in X\times Y$ such that $\chi_VT\chi_U \not=0$ for all open neighbourhoods $U$ of $x$ and $V$ of $y$, where $\chi_U$ and $\chi_V$ are the characteristic functions of $U$ and $V$, respectively.
\end{definition}

    Please note that the support defined in \cite{WillettYuBook} is the inverse of ours.

    We give some properties of the support, the proof can be obtained similarly from chapter 4 of \cite{WillettYuBook}.
\begin{remark}\label{Remark:support}
Let $E_X^p$  be an $L^p$-$X$-module, $E_Y^p$  be an $L^p$-$Y$-module, and $E_Z^p$  be an $L^p$-$Z$-module. Let $R,S:E_X^p\rightarrow E_Y^p$ and $T:E_Y^p\rightarrow E_Z^p$ be bounded linear operators. Then:
  \begin{enumerate}
  \item $\mathrm{supp}(R+S) \subseteq \mathrm{supp}(R) \cup \mathrm{supp}(S)$;
  \item $\mathrm{supp}(TS) \subseteq \mathrm{cl}(\mathrm{supp}(S) \circ \mathrm{supp}(T))=\mathrm{cl}(\{(x,z)\in X\times Z: \exists\, y\in Y \text{ s.t. } (x,y)\in \mathrm{supp}(S)$, $(y,z)\in \mathrm{supp}(T)\})$, where `$\mathrm{cl}$' means closure;
  \item If the coordinate projections $\pi_Y$: $\mathrm{supp}(T)\rightarrow Y$ and $\pi_Z: \mathrm{supp}(T)\rightarrow Z$ are proper maps, or coordinate projections $\pi_X: \mathrm{supp}(S)\rightarrow X$ and $\pi_Y: \mathrm{supp}(S)\rightarrow Y$ are proper maps, then $\mathrm{supp}(TS) \subseteq \mathrm{supp}(S) \circ \mathrm{supp}(T)$;
  \item Let $F$ = supp($S$), then for any compact subset $K$ of $X$, respectively $Y$, we have $S\chi_K=\chi_{K\circ F}S\chi_K$, $\chi_KS=\chi_KS\chi_{F\circ K}$, where $K\circ F:=$\{$y\in Y:$ there is $x\in K$ such that $(x,y)\in F$\}, $F\circ K:=$\{$x\in X:$ there is $y\in K$ such that $(x,y)\in F$\}.
  \end{enumerate}
\end{remark}

\begin{definition}
Let $E_X^p$ be an $L^p$-$X$-module and $T$ be a bounded linear operator acting on $E_X^p$.
  \begin{enumerate}
  \item The \textit{propagation} of $T$ is defined to be $\mathrm{prop}(T)=\sup\{d(x,y):(x,y)\in \mathrm{supp}(T)\}$;
  \item $T$ is said to be \textit{locally compact} if $\chi_K T$ and $T \chi_K$ are compact operators for all compact subset $K$ of $X$.
  \end{enumerate}
\end{definition}
    By Remark \ref{Remark:support}, We have the following properties of propagation.
\begin{remark}\label{Remark:propagation}
Let $E_X^p$ be an $L^p$-$X$-module and let $T,S:E_X^p\rightarrow E_X^p$ be bounded linear operators. Then:
  \begin{enumerate}
  \item $\mathrm{prop}(T+S) \leq \max\{\mathrm{prop}(T),\mathrm{prop}(S)\}$;
  \item $\mathrm{prop}(TS) \leq \mathrm{prop}(T) + \mathrm{prop}(S)$.
  \end{enumerate}
\end{remark}

\begin{definition}
Let $E_X^p$ be an $L^p$-$X$-module. The $L^p$ Roe algebra of $E_X^p$, denoted $B^p(E_X^p)$, is defined to be the norm closure of the algebra of all locally compact operators acting on $E_X^p$ with finite propagations.
\end{definition}

A Borel map $f$ from a proper metric space $X$ to another proper metric space $Y$ is called \textit{coarse} if (1) $f$ is proper, i.e., the inverse image of any bounded set is bounded; (2) for every $R>0$, there exists $R'>0$ such that $d(f(x),f(y))\leq R'$ for all $x,y\in X$ satisfying $d(x,y)\leq R$.

\begin{lemma}\label{Lemma:CoveringIsometry}
Let $f$ be a continuous coarse map, let $E_X^p$ be an $L^p$-$X$-module and  $E_Y^p$ be an $L^p$-$Y$-module. Then for any $\varepsilon>0$, there exist an isometric operator $V_f:E_X^p\rightarrow E_Y^p$ and a contractive operator $V_f^+:E_Y^p\rightarrow E_X^p$ with $V_f^+V_f=I$ such that
  \begin{center}
  $\mathrm{supp}(V_f) \subseteq \{(x,y)\in X\times Y:d(f(x),y)\leq \varepsilon\}$ \\
  $\mathrm{supp}(V_f^+) \subseteq\{(y,x)\in Y\times X:d(f(x),y)\leq \varepsilon\}$
  \end{center}
\end{lemma}
\begin{proof}
    Let $Z_X$, $Z_Y$ be the dense subsets of $X$ and $Y$ for defining $E^p_X$ and $E^p_Y$, respectively, as in Definition \ref{Def:Lp-module}

        There exists a Borel cover $\{Y_i\}_i$ of $Y$ such that:
         \begin{enumerate}
         \item  $Y_i\cap Y_j=\emptyset$ if $i\not=j$;
         \item $\mathrm{diameter}(Y_i)\leq \varepsilon$ for all $i$;
         \item each $Y_i$ has nonempty interior.
    \end{enumerate}
     Condition (3) implies that $Y_i\cap Z_Y$ is a countable set. Thus if $f^{-1}(Y_i)\cap Z_X\not=\emptyset$, then there exists an isometric operator $V_i:\ell^p(f^{-1}(Y_i)\cap Z_X)\otimes \ell^p\rightarrow \ell^p(Y_i\cap Z_Y)\otimes \ell^p$ and a contractive operator $V_i^+:\ell^p(Y_i\cap Z_Y)\otimes \ell^p\rightarrow \ell^p(f^{-1}(Y_i)\cap Z_X)\otimes \ell^p$ such that $V_i^+V_i=\chi_{f^{-1}(Y_i)\cap Z_X}\otimes I$. If $f^{-1}(Y_i)\cap Z_X=\emptyset$, then let $V_i=V_i^+=0$. Define
$$V_f=\bigoplus_{i} V_i:\bigoplus_{i} \ell^p(f^{-1}(Y_i)\cap Z_X)\otimes \ell^p\rightarrow \bigoplus_{i} \ell^p(Y_i\cap Z_Y)\otimes \ell^p$$
$$V_f^+=\bigoplus_{i} V_i^+:\bigoplus_{i} \ell^p(Y_i\cap Z_Y)\otimes \ell^p \rightarrow \bigoplus_{i} \ell^p(f^{-1}(Y_i)\cap Z_X)\otimes \ell^p.$$
    Then $V_f$ is an isometric operator, $V_f^+$ is a contractive operator and $V_f^+V_f=I$.
Condition (2) together with the construction of $V_f$ and $V_f^+$, implies that $\mathrm{supp}(V_f) \subseteq \{(x,y)\in X\times Y:d(f(x),y)\leq \varepsilon\}$ and $\mathrm{supp}(V_f^+) \subseteq\{(y,x)\in Y\times X:d(f(x),y)\leq \varepsilon\}$.
\end{proof}

\begin{lemma}\label{Lemma:CoveringIsometryPair}
Let $f$, $E_X^p$ and $E_Y^p$ be as in Lemma \ref{Lemma:CoveringIsometry}. Then the pair $(V_f,V_f^+)$ gives rise to a homomorphism $\mathrm{ad}((V_f,V_f^+)):B^p(E_X^p)\rightarrow B^p(E_Y^p)$ defined by:
                        $$\mathrm{ad}((V_f,V_f^+))(T)=V_fTV_f^+$$ for all $T\in B^p(E_X^p)$.\par
Moreover, the map $\mathrm{ad}((V_f,V_f^+))_*$ induced by $\mathrm{ad}((V_f,V_f^+))$ on $K$-theory depends only on $f$ and not on the choice of the pair $(V_f,V_f^+)$.
\end{lemma}
\begin{proof}
    Obviously, $\mathrm{ad}(V_f,V_f^+)$ be a contractive homomorphism, thus we just need to show that if $T$ has finite propagation and is locally compact, then $\mathrm{ad}((V_f,V_f^+))(T)$ has these properties too.\par
    Assume first that $T$ has finite propagation. Let $\varepsilon$ be as in Lemma \ref{Lemma:CoveringIsometry}, then $d(f(x),y)\leq \varepsilon$ and $d(f(x'),y')\leq \varepsilon$ for any $(x,y)\in \mathrm{supp}(V_f)$ and $(y',x')\in \mathrm{supp}(V_f^+)$. Let $(y_1,y_2)\in \mathrm{supp}(V_fTV_f^+)$. By Remark \ref{Remark:support} part (3), we have that
  \begin{center}
    $\mathrm{supp}(V_fTV_f^+) \subseteq \mathrm{supp}(V_f) \circ \mathrm{supp}(T)\circ \mathrm{supp}(V_f^+)$.
  \end{center}
Hence there exist $x_1,x_2\in X$ such that $(x_1,y_1)\in \mathrm{supp}(V_f), (x_1,x_2)\in \mathrm{supp}(T)$ and $(y_2,x_2)\in \mathrm{supp}(V_f^+)$,
then
    $$d(y_1,y_2)\leq d(y_1,f(x_1))+d(f(x_1),f(x_2))+d(f(x_2),y_2)\leq 2\varepsilon +d(f(x_1),f(x_2)).$$
Since $f$ is coarse and $T$ has finite propagation, we have that $d(y_1,y_2)$ is smaller than some constant for all $(y_1,y_2)\in \mathrm{supp}(V_fTV_f^+)$. This completes the proof of finite propagation.\par
    Now assume that $T$ is locally compact. Let $K$ be a compact subset of $Y$, and let $F = \mathrm{supp}(V_f)$. By Remark \ref{Remark:support} (4), we have that$$\chi_K V_fTV_f^+=\chi_K V_f\chi_{F\circ K}TV_f^+$$
 Since $f$ is a proper map and $X$ is a proper space, we know that $F\circ K$ is a compact subset in $X$, then $\chi_{F\circ K}T$ is compact operator, thus $\chi_K V_f\chi_{F\circ K}TV_f^+$ is a compact operator. The case of $V_fTV_f^+\chi_K$ is similar. Thus $\mathrm{ad}((V_f,V_f^+))(T)$ is locally compact.\par
    Let $(V_1,V_1^+)$ and $(V_2,V_2^+)$ be two pairs of operators satisfying the conditions of Lemma \ref{Lemma:CoveringIsometry}, then we just need to prove
               $$\mathrm{ad}((V_1,V_1^+))_*=\mathrm{ad}((V_2,V_2^+))_*: K_*(B^p(E_X^p))\rightarrow K_*(B^p(E_Y^p))$$
Let
  \begin{center}
  \begin{equation*}
  U=\begin{pmatrix}
  I-V_1V_1^+ & V_1V_2^+ \\
  V_2V_1^+   & I-V_2V_2^+
  \end{pmatrix}
  \end{equation*}
  \end{center}
then $U^2=I$ and
 \begin{center}
 \begin{equation*}
  \begin{pmatrix}
 \mathrm{ad}((V_1,V_1^+))(T) & 0\\
 0                  & 0
  \end{pmatrix}
  =U\begin{pmatrix}
 0 & 0\\
 0 & \mathrm{ad}((V_2,V_2^+))(T)
  \end{pmatrix}U
  \end{equation*}
 \end{center}
Thus $\mathrm{ad}((V_1,V_1^+))_*=\mathrm{ad}((V_2,V_2^+))_*$
\end{proof}

\begin{corollary}
    For different $L^p$-$X$-modules $E_X^p$ and $E_X'^{p}$, $B^p(E_X^p)$ is non-canonically isomorphic to $B^p(E_X'^{p})$, and $K_*(B^p(E_X^p))$ is canonically isomorphic to $K_*(B^p(E_X'^{p}))$.
\end{corollary}
For convenience, we replace $B^p(E_X^p)$ by $B^p(X)$ representing $L^p$ Roe algebra of $X$.

\subsection{$L^p$ Localization algebra and $L^p$ $K$-homology}

\begin{definition}
   Let $X$ be a proper metric space. The \textit{$L^p$ localization algebra} of $X$, denoted $B_L^p(X)$, is defined to be the norm closure of the algebra of all bounded and uniformly norm-continuous functions $f$ from $[0,\infty)$ to $B^p(X)$ such that
  \begin{center}
   prop($f(t)$) is uniformly bounded and prop($f(t)$)$\rightarrow 0$ as $t\rightarrow \infty$.
  \end{center}
  The \textit{propagation} of $f$ is defined to be $\max\{\mathrm{prop}(f(t)):t\in [0,\infty)\}$.
\end{definition}

    Let $f$ be a uniformly continuous coarse map from a proper metric space $X$ to another proper metric space $Y$. Let $\{\varepsilon_k\}_k$ be a sequence of positive numbers such that $\varepsilon_k \rightarrow 0$ as $k \rightarrow \infty$. By Lemma \ref{Lemma:CoveringIsometry}, for each $\varepsilon_k$, there exists an isometric operator $V_k$ from an $L^p$-$X$-module $E_X^p$ to an $L^p$-$Y$-module $E_Y^p$ and a contractive operator $V_k^+$ from an $L^p$-$Y$-module $E_Y^p$ to an $L^p$-$X$-module $E_X^p$ such that $V_k^+V_k=I$ and
    \begin{center}
    $\mathrm{supp}(V_k) \subseteq\{(x,y)\in X\times Y : d(f(x),y)\leq \varepsilon_k\}$\\
    $\mathrm{supp}(V_k^+) \subseteq\{(y,x)\in Y\times X : d(f(x),y)\leq \varepsilon_k\}$.
    \end{center}
For $t\in [0,\infty)$, define
    \begin{center}
    $V_f(t)=R(t-k)(V_k\oplus V_{k+1})R^*(t-k)$\\
    $V_f^+(t)=R(t-k)(V_k^+\oplus V_{k+1}^+)R^*(t-k)$
    \end{center}
for all $k\leq t \leq k+1$, where
    $$R(t)=\begin{pmatrix}
    \cos(\pi t/2)  & \sin(\pi t/2)\\
    -\sin(\pi t/2) & \cos(\pi t/2)
    \end{pmatrix}.$$
 $V_f(t)$ is an operator from $E_X^p \oplus E_X^p$ to $E_Y^p \oplus E_Y^p$, and $V_f^+(t)$ is an operator from $E_Y^p \oplus E_Y^p$ to $E_X^p \oplus E_X^p$ such that $||V_f(t)||\leq 4$, $||V^+_f(t)||\leq 4$ and $V_f^+(t)V_f(t)=I$ for all $t\in[0,\infty)$.

\begin{lemma}\label{Lemma:IsometryPairHom}
Let $f$ and $\{\varepsilon_k\}_k$ be as above, then the pair $(V_f(t),V_f^+(t))$ induces a homomorphism $\mathrm{Ad}((V_f,V_f^+))$ from $B_L^p(X)$ to $B_L^p(Y)\otimes M_2(\mathbb{C})$ defined by:
$$\mathrm{Ad}((V_f,V_f^+))(u)(t)=V_f(t)(u(t)\oplus 0)V_f^+(t)$$
for any $u\in B_L^p(X)$ and $t\in [0,\infty)$, such that
  $$\mathrm{prop}(\mathrm{Ad}((V_f,V_f^+))(u)(t))\leq \sup_{(x,y)\in \mathrm{supp}(u(t))}d(f(x),f(y))+2\varepsilon_k + 2\varepsilon_{k+1}$$
for all $t\in [k,k+1]$.\par
Moreover, the induced map $\mathrm{Ad}((V_f,V_f^+))_*$ on $K$-theory depends only on f and not on the choice of the pairs $\{(V_k,V_k^+)\}$ in the construction of $V_f(t)$ and $V_f^+(t)$.
\end{lemma}
\begin{proof}
For any $u\in B_L^p(X)$, $\mathrm{Ad}((V_f,V_f^+))(u)$ is bounded and uniformly norm-continuous in $t$ although $V_f$ and $V_f^+$ are not norm-continuous. By the same ways as the proof of Lemma \ref{Lemma:CoveringIsometryPair}, we can obtain that $\mathrm{Ad}((V_f,V_f^+))(u)(t)$ is locally compact when $u(t)$ is locally compact for each $t$ and $\mathrm{Ad}((V_f,V_f^+))_*$ does not depend on the choice of the pair $(V_f,V_f^+)$.\par
Thus we just need to consider prop($\mathrm{Ad}((V_f,V_f^+))(u)(t)$) for which prop($u(t)$) is uniformly finite and prop($u(t)$)$\rightarrow 0$ as $t\rightarrow \infty$. By Lemma \ref{Remark:support} (4), we know that
\begin{align*}
  \mathrm{prop}(V_k u(t) V_k^+) &\le\sup\{d(f(x),f(y)):(x,y)\in \mathrm{supp}(u(t))\}+2\varepsilon_k\\
  \mathrm{prop}(V_k u(t) V_{k+1})& \le \sup\{d(f(x),f(y)):(x,y)\in \mathrm{supp}(u(t))\}+\varepsilon_k+\varepsilon_{k+1}.
\end{align*}
Thus by Remark \ref{Remark:propagation}, we have
  \begin{center}
  prop$(\mathrm{Ad}((V_f,V_f^+))(u)(t))$ $\leq$ sup$\{d(f(x),f(y)):(x,y)\in \mathrm{supp}(u(t))\}$ $+$ $2\varepsilon_k + 2\varepsilon_{k+1}$
  \end{center}
Therefore, $\mathrm{prop}(\mathrm{Ad}((V_f,V_f^+))(u)(t))$ is uniformly finite since $f$ is a coarse map, and $\mathrm{prop}(\mathrm{Ad}((V_f,V_f^+))(u)(t)$)$\rightarrow 0$ as $t\rightarrow \infty$ since $f$ is a uniformly continuous map and $\varepsilon_k\rightarrow 0$.
\end{proof}

\begin{definition}
    The $i$-th \textit{$L^p$ $K$-homology} of $X$, is defined to be $K_i(B^p_L(X))$.
\end{definition}

\subsection{Obstruction group}

    Let $X$ be a proper metric space. Now consider the evaluation-at-zero homomorphism:
    $$e_0:B_L^p(X)\rightarrow B^p(X)$$
which induces a homomorphism on $K$-theory:
    $$e_0:K_*(B_L^p(X))\rightarrow K_*(B^p(X))$$

    Let $C$ be a locally finite and uniformly bounded cover for  $X$. The \textit{nerve space} $N_C$ associated to $C$ is defined to be the simplicial complex whose set of vertices equals $C$ and where a finite subset $\{U_0,\ldots,U_n\}\subseteq C$ spans an $n$-simplex in $N_C$ if and only if $\bigcap_{i=0}^n U_i \not=\emptyset$. Endow $N_C$ with the \textit{$\ell^1$-metric}, i.e., the path metric whose restriction to each simplex $\{U_0,\ldots,U_n\}$ is given by $$d(\sum^n_{i=0}t_iU_i,\sum^n_{i=0}s_iU_i)=\sum^n_{i=0}|t_i-s_i|.$$ The distance of two points which in different connected components is defined to be $\infty$ by convention.

    A sequence of locally finite and uniformly bounded covers $\{C_k\}_{k=0}^\infty$ of metric space $X$ is called an \textit{anti-\v Cech system} of $X$ \cite{RoeBookCoarseCohomology}, if there exists a sequence of positive numbers $R_k\rightarrow \infty$ such that for each $k$,\par
   (1) every set $U\in C_k$ has diameter less than or equal to $R_k$;\par
   (2) any set of diameter $R_k$ in $X$ is contained in some member of $C_{k+1}$.

    An anti-\v Cech system always exists \cite{RoeBookCoarseCohomology}.     

    By the property of the anti-\v Cech system, for every pair $k_2>k_1$, there exists a simplicial map $i_{k_1 k_2}$ from $N_{C_{k_1}}$ to $N_{C_{k_2}}$ such that $i_{k_1k_2}$ maps a simplex $\{U_0,\ldots,U_n\}$ in $N_{C_{k_1}}$ to a simplex $\{U_0',\ldots,U_n'\}$ in $N_{C_{k_2}}$ satisfying $U_i\subseteq U_i'$ for all $0\leq i\leq n$. Thus, $i_{k_1 k_2}$ gives rise to the following inductive systems of groups:
    \begin{center}
    $\mathrm{ad}((V_{i_{k_1k_2}},V_{i_{k_1k_2}}^+))_*:K_*(B^p(N_{C_{k_1}}))\rightarrow K_*(B^p(N_{C_{k_2}}))$;\\
    $\mathrm{Ad}((V_{i_{k_1k_2}},V_{i_{k_1k_2}}^+))_*:K_*(B_L^p(N_{C_{k_1}}))\rightarrow K_*(B_L^p(N_{C_{k_2}}))$.
    \end{center}

    The following conjecture is called the $L^p$ coarse Baum-Connes conjecture.
\begin{conjecture}
Let $X$ be a proper metric space, $\{C_k\}_{k=0}^\infty$ be an anti-\v Cech system of $X$, then the evaluation-at-zero  homomorphism
$$e_0:\lim_{k\rightarrow \infty} K_*(B_L^p(N_{C_k})) \rightarrow \lim_{k\rightarrow \infty} K_*(B^p(N_{C_k})) \cong K_*(B^p(X))$$
is an isomorphism.
\end{conjecture}
    For each $p\in [1,\infty)$, the group $\lim_{k\rightarrow \infty} K_*(B_L^p(N_{C_k}))$ is the \textit{$L^p$ coarse $K$-homology of $X$} (refer Definition 2.1 in \cite{HigsonRoeCBC}). Moreover, it is not difficult to see that the $L^p$ coarse Baum-Connes conjecture for $X$ does not depend on the choice of the anti-\v Cech system.

    Let $B_{L,0}^p(X)=\{f\in B_L^p(X):f(0)=0\}$. There exists an exact sequence:
    $$0 \rightarrow B_{L,0}^p(X) \rightarrow B_L^p(X) \rightarrow B^p(X) \rightarrow 0$$
    Thus we have the following reduction:
\begin{lemma}\label{Thm:VanishingObstruction}
Let $X$ be a proper metric space, $\{C_k\}_{k=0}^\infty$ be an anti-\v Cech system of $X$, then the $L^p$ coarse Baum-Connes conjecture is true if and only if
         $$\lim_{k\rightarrow \infty} K_*(B_{L,0}^p(N_{C_k}))=0$$
\end{lemma}
    For obvious reason $\lim_{k\rightarrow \infty} K_*(B_{L,0}^p(N_{C_k}))$ is called the obstruction group to the $L^p$ coarse Baum-Connes conjecture.

\section{Controlled obstructions: $QP_{\delta,N,r,k}(X),QU_{\delta,N,r,k}(X)$}
    The controlled obstruction $QP$ and $QU$ for the coarse Baum-Connes conjecture was introduced by Guoliang Yu \cite{YuFAD}.
    In this section, we will introduce and study the $L^p$ version of $QP$ and $QU$, which can be considered as a controlled version of $K_0(B^p_{L,0}(X)\otimes C_0((0,1)^k))$ and $K_1(B^p_{L,0}(X)\otimes C_0((0,1)^k))$. We will follow the notation in \cite{YuFAD}. One may refer to \cite{OyonoYuQuantitative}\cite{ChungQuantiativeBanach} for more detail about the controlled $K$-theory for $C^\ast$-algebras and $L^p$-algebras.

\subsection{Fundamental concept and property}

\begin{definition}(\cite{ChungQuantiativeBanach})
    Let A be a unital Banach algebra, for $0<\delta<1/100,N\geq 1$, we define
(1) an element $e$ in $A$ is called \textit{$(\delta,N)$-idempotent}, if $||e^2-e||<\delta$ and max$\{||e||,||I-e||\}\leq N$;
(2) an element $u$  in $A$ is called $(\delta,N)$-$invertible$, if $||u||\leq N$, and there exists $v\in A$ with $||v||\leq N$ such that max$\{||uv-I||,||vu-I||\}<\delta$, where $I$ is the unit of $A$. Such $v$ is called a $(\delta,N)$-inverse of $u$.
\end{definition}

    Let $X$ be a proper metric space, let $B^p_{L,0}(X)^+$ be the Banach algebra obtained from $B^p_{L,0}(X)$ by adjoining an identity $I$.

\begin{definition}
    Let $0<\delta<1/100, N\geq 1, r>0$, $k$ and $n$ be nonnegative integers. Define $QP_{\delta,N,r,k}(B^p_{L,0}(X)^+ \otimes M_n(\mathbb{C}))$ to be the set of all continuous functions $f$ from $[0,1]^k$ to $B^p_{L,0}(X)^+ \otimes M_n(\mathbb{C})$ such that:\par
(1) $f(t)$ is an $(\delta,N)$-idempotent and prop$(f(t))$ $\leq r$ for all $t\in [0,1]^k$;\par
(2) $||f(t)-e_m||<\delta$ for all $t\in \mathrm{bd}([0,1]^k)$, the boundary of $[0,1]^k$ in $\mathbb{R}^k$, where $e_m=I\oplus \ldots \oplus I\oplus 0\oplus \ldots \oplus 0$ with $m$ identities;\par
(3) $\pi(f(t))=e_m$, where $\pi$ is the canonical homomorphism from $B^p_{L,0}(X)^+ \otimes M_n(\mathbb{C})$ to $M_n(\mathbb{C})$.
\end{definition}
\begin{definition}
    Let $0<\delta<1/100, N\geq 1, r>0$, $QP_{\delta,N,r,k}(X)$ is defined to be the direct limit of $QP_{\delta,N,r,k}(B^P_{L,0}(X)^+ \otimes M_n(\mathbb{C}))$ under the embedding: $p\rightarrow p\oplus 0$.
\end{definition}

\begin{definition}
    Let $0<\delta<1/100, N\geq 1, r>0$, $k$ and $n$ be nonnegative integers. Define $QU_{\delta,N,r,k}(B^p_{L,0}(X)^+ \otimes M_n(\mathbb{C}))$ to be the set of all continuous functions $u$ from $[0,1]^k$ to $B^p_{L,0}(X)^+ \otimes M_n(\mathbb{C})$ such that there exists a continuous function $v:[0,1]^k \rightarrow B^p_{L,0}(X)^+ \otimes M_n(\mathbb{C})$ satisfying that:\par
(1) $u(t)$ is a $(\delta,N)$-invertible with a $(\delta,N)$-inverse $v(t)$ such that $\max\{\mathrm{prop}(u(t))$, $\mathrm{prop}(v(t))$\} $\leq r$ for all $t\in [0,1]^k$;\par
(2) $||u(t)-I||<\delta$ and $||v(t)-I||<\delta$ for all $t\in \mathrm{bd}([0,1]^k)$;\par
(3) $\pi(u(t))=\pi(v(t))=I$, where $\pi$ is the canonical homomorphism from $B^p_{L,0}(X)^+ \otimes M_n(\mathbb{C})$ to $M_n(\mathbb{C})$.\par
Such $v$ is called a \textit{$(\delta,N,r)$-inverse} of $u$.
\end{definition}
\begin{definition}
    Let $0<\delta<1/100, N\geq 1, r>0$, $QU_{\delta,N,r,k}(X)$ is defined to be the direct limit of $QU_{\delta,N,r,k}(B^P_{L,0}(X)^+ \otimes M_n(\mathbb{C})$ under the embedding: $u\rightarrow u\oplus I$.
\end{definition}

\begin{definition}
    Let $e_1,e_2\in QP_{\delta,N,r,k}(B^p_{L,0}(X)^+ \otimes M_n(\mathbb{C}))$. We say $e_1$ is \textit{$(\delta,N,r)$-equivalent} to $e_2$, if there exists a continuous homotopy $a(t')$ in $QP_{\delta,N,r,k}(B^p_{L,0}(X)^+ \otimes M_n(\mathbb{C}))$ for $t'\in [0,1]$,
such that $a(0)=e_1$ and $a(1)=e_2$. Such homotopy is called a \textit{$(\delta,N,r)$-homotopy}.
\end{definition}

Notice that (1) any $e\in QP_{\delta,N,r,k}(X)$ is $(\delta',N',r)$-equivalent to some $f$ for which $f(t)=\pi(f)$ for all $t\in \mathrm{bd}([0,1]^k)$; (2) if $e_1$ is $(\delta,N,r)$-equivalent to $e_2$ and $e_1(t)=\pi(e_1)$, $e_2(t)=\pi(e_2)$ for all $t\in \mathrm{bd}([0,1]^k)$, then there exists a homotopy $a(t')$ in $QP_{\delta'',N'',r,k}(X)$ such that $a(0)=e_1,a(1)=e_2$ and $a(t')(t)=\pi(a(t'))$ for all $t\in \mathrm{bd}([0,1]^k)$, where $\delta',\delta''$ depend only on $\delta,N$; $N',N''$ depend only on $N$.

\begin{definition}
    Let $u_1,u_2$ are two elements in $QU_{\delta,N,r,k}(B^p_{L,0}(X)^+ \otimes M_n(\mathbb{C}))$, we say $u_1$ is $(\delta,N,r)$-equivalent to $u_2$, if there exists a continuous homotopy $w(t')$ in $QU_{\delta,N,r,k}(B^p_{L,0}(X)^+ \otimes M_n(\mathbb{C}))$ for $t'\in [0,1]$ such that $w(0)=u_1$ and $w(1)=u_2$. This equivalence induce an equivalent relation in $QU_{\delta,N,r,k}(X)$.
\end{definition}

    The following lemma tells us that $QP_{\delta,N,r,k}(X)$ can be considered as a controlled version of $K_0(B^p_{L,0}(X) \otimes C_0((0,1)^k))$.
\begin{lemma}\label{Lemma:QuasiVSTrueProjection}
    Let $0<\delta<1/100$ and $\chi$ is a function such that $\chi(x)=1$ for $\mathrm{Re}(x)>1/2$; $\chi(x)=0$ for $\mathrm{Re}(x)<1/2$,\par
(1) for any $e\in QP_{\delta,N,r,k}(X)$, $\chi(e)$ is an idempotent and defines an element $[\chi(e)]\in K_0(B^p_{L,0}(X)\otimes C_0((0,1)^k))$;\par
(2) for any two elements $e_1,e_2\in QP_{\delta,N,r,k}(X)$ satisfying that $e_1$ is $(\delta,N,r)$-equivalent to $e_2$, then $[\chi(e_1)]=[\chi(e_2)]$ in $K_0(B^p_{L,0}(X)\otimes C_0((0,1)^k))$;\par
(3) for any $0<\delta<1/100$, every element in $K_0(B^p_{L,0}(X)\otimes C_0((0,1)^k))$ can be represented as $[\chi(e_1)]-[\chi(e_2)]$, where $e_1,e_2\in QP_{\delta,N,r,k}(X)$ for some $N\geq 1$ and $r>0$.
\end{lemma}

\begin{proof}
    (1) and (2) are straightforward by holomorphic functional calculus and the definition of $(\delta,N,r)$-equivalence. To prove (3), for any $[p]-[q]\in K_0(B^p_{L,0}(X)\otimes C_0((0,1)^k))$ where $p,q\in (B^p_{L,0}\otimes C_0((0,1)^k))^+ \otimes M_n(\mathbb{C})$, let $N=||p||+||1-p||+1$.
     By approximation argument there exists $r>0$ and $e_1\in (B^p_{L,0}\otimes C((0,1)^k) )^+\otimes M_n(\mathbb{C})$ such that $\mathrm{prop}(e_1)<r$ and $||e_1-p||<\frac{\delta}{4N}$. Thus we have $e_1\in QP_{\delta,N,r,k}(X)$. Now we just need to prove $[\chi(e_1)]=[p]$. Let $e(t')=t'e_1+(1-t')p$ for $t'\in [0,1]$. We have $||e^2(t')-e(t')||<\delta$. Thus $\chi(e(t'))$ is a continuous homotopy of projections between $\chi(e(0))=p$ and $\chi(e(1))=\chi(e_1)$.
\end{proof}

    The following lemma tells us that $QU_{\delta,N,r,k}(X)$ can be considered as a controlled version of $K_1(B^p_{L,0}(X) \otimes C_0((0,1)^k))$.
\begin{lemma}\label{Lemma:QuasiVSTrueInvertible}
    Let $0<\delta<1/100$,\par
(1) for any $u\in QU_{\delta,N,r,k}(X)$, $u$ is an invertible element and defines an element $[u]$ in $K_1(B^p_{L,0}(X)\otimes C_0((0,1)^k))$;\par
(2) if $u_1$ is $(\delta,N,r)$-equivalent to $u_2$ in $QU_{\delta,N,r,k}(X)$, then $[u_1]=[u_2]$ in $K_1(B^p_{L,0}(X)\otimes C_0((0,1)^k))$;\par
(3) for any $0<\delta<1/100$, every element in $K_1(B^p_{L,0}(X)\otimes C_0((0,1)^k))$ can be represented as $[u]$, where $u\in QU_{\delta,N,r,k}(X)$ for some $N\geq 1$ and $r>0$.
\end{lemma}

\begin{proof}
    (1) is true since the set of invertible elements in Banach algebra is open.\par
     (2) is true by the definition of $(\delta,N,r)$-equivalence. \par
     To prove (3), we assume $[u']\in K_1(B^p_{L,0}(X)\otimes C_0((0,1)^k))$. Let $N=||u'||+||u'^{-1}||+1$, then there exists $r>0$ and $u,v\in B^p_{L,0}(X)\otimes C_0((0,1)^k)\otimes M_n(\mathbb{C})$ such that $||u-u'||<\frac{\delta}{2N}, ||v-u'_{-1}||<\frac{\delta}{2N}$ and $\mathrm{prop}(u)<r,\mathrm{prop}(v)<r$. We have $||u||\leq N, ||v||\leq N$ and $||uv-I||<\delta,||vu-I||<\delta$. Hence $u\in QU_{\delta,N,r,k}(X)$. Let $w(t)=tu+(1-t)u'$ for $t\in [0,1]$. We have $||w(t)u'^{-1}-I||<\delta<1/100$. Thus $w(t)u'^{-1}$ is an invertible element, so is $w(t)$. Therefore $[u]=[u']$ in $K_1(B^p_{L,0}(X)\otimes C_0((0,1)^k))$.
\end{proof}

\begin{lemma}\label{Lemma:HomotopyImpliesLipschitz}(Lemma 2.29 in \cite{ChungQuantiativeBanach})   
    If $e$ is $(\delta,N,r)$-equivalent to $f$ by a homotopy $e_{t'}(t'\in [0,1])$ in $QP_{\delta,N,r,k}(X)$, then there exists $\alpha_N>0,m\in \mathbb{N}$ such that $e\oplus I_m\oplus 0_m$ is $(2\delta,3N,r)$-equivalent to $f\oplus I_m\oplus 0_m$ by a $\alpha_N$-Lipschitz homotopy, where $\alpha_N$ depends only on $N$ and not on $e,f,\delta,r$; and $m$ depends only on $\delta,N,e_{t'}$.
\end{lemma}

\begin{proof}
   There exists a partition $0=t'_0<t'_1<\ldots<t'_m=1$ such that
        $$||e_{t'_i}-e_{t'_{i-1}}||<\inf_{t'\in [0,1]} \frac{\delta-||e^2_{t'}-e_{t'}||}{2N+1}.$$
For each $t'$, we have a Lipschitz $(\delta,3N,r)$-homotopy between $I\oplus 0$ and $e_{t'}\oplus (1-e_{t'})$ given by combining the linear homotopy connecting $I\oplus 0$ to $(e_{t'}-e^2_{t'})\oplus 0$ and the homotopy
      $$(e_{t'}\oplus 0)+R^*(s)((1-e_{t'})\oplus 0)R(s),$$
where $R(s)=\begin{pmatrix}
        \cos(\pi s/2) & \sin(\pi s/2)\\
        -\sin(\pi s/2) & \cos(\pi s/2)
        \end{pmatrix}$.
Obviously, the linear homotopy between $e_{t'_{i-1}}$ and $e_{t'_i}$ is Lipschitz for all $i$. Then\par
     $\begin{pmatrix}
     e_{t'_0} &     &    \\
              & I_m &    \\
              &     &   0_m
      \end{pmatrix}\\
\simeq \begin{pmatrix}
 e_{t'_0} &   &   &        &   &\\
          & I &   &        &   &\\
          &   & 0 &        &   &\\
          &   &   & \ddots &   &\\
          &   &   &        & I & \\
          &   &   &        &   & 0
 \end{pmatrix}
\simeq \begin{pmatrix}
 e_{t'_0} &            &            &        &            &\\
          & I-e_{t'_1} &            &        &            &\\
          &            & e_{t'_1}   &        &            &\\
          &            &            & \ddots &            &\\
          &            &            &        & I-e_{t'_m} & \\
          &            &            &        &            & e_{t'_m}
 \end{pmatrix}\\
\simeq \begin{pmatrix}
 e_{t'_0} &            &            &        &                &\\
          & I-e_{t'_0} &            &        &                &\\
          &            & e_{t'_1}   &        &                &\\
          &            &            & \ddots &                &\\
          &            &            &        & I-e_{t'_{m-1}} & \\
          &            &            &        &                & e_{t'_m}
 \end{pmatrix}
\simeq \begin{pmatrix}
        I &   &   &        &   &\\
          & 0 &   &        &   &\\
          &   & I &        &   &\\
          &   &   & \ddots &   &\\
          &   &   &        & 0 & \\
          &   &   &        &   & e_{t'_m}
 \end{pmatrix}\\
\simeq \begin{pmatrix}
        e_{t'_m} &     &    \\
                 & I_m &    \\
                 &     & 0_m
      \end{pmatrix}$.\\
Where $\simeq$ represent $(2\delta,3N,r)$-equivalence by Lipschitz homotopy.
   \end{proof}

We remark that we have a result for $QU$ similar to the above lemma, i.e., homotopy implies Lipschitz homotopy. \par

The following lemma tells us that homotopy equivalence of two quasi-invertible elements implies  homotopy equivalence of their quasi-inverses.

\begin{lemma}\label{Lemma:EquivImplyInverseEquiv}
    Let $u_1,u_2$ be two elements in $QU_{\delta,N,r,k}(X)$ with $(\delta,N)$-inverse $v_1,v_2$ respectively, if $u_1$ is $(\delta,N,r)$-equivalent to $u_2$, then $v_1$ is $(4\delta,2N,r)$-equivalent to $v_2$ in $QU_{4\delta,2N,r,k}(X)$.
\end{lemma}

\begin{proof}
    Let $w(t')$ be the homotopy path jointing $u_1$ and $u_2$, for $\varepsilon=\frac{\delta}{N}$, there exists a partition $0=t'_0<t'_1<\ldots<t'_n=1$ such that
      $$\max_{0\leq i \leq n-1} \left\{||w(l)-w(l')||:t'_i\leq l,l' \leq t'_{i+1} \right\}<\frac{\delta}{N}.$$
Assume $s_{t'_i}$ is the $(\delta,N,r)$-inverse of $w(t'_i)$, we require $s_0=v_1,s_1=v_2$. Let
      $$s(t')=\frac{t'-t'_i}{t'_{i+1}-t'_i}s_{t'_{i+1}}-\frac{t'-t'_{i+1}}{t'_{i+1}-t'_i}s_{t'_i},t'_i\leq t' \leq t'_{i+1}$$
We have $||s_{t'_i}w(t')-I||\leq ||s_{t'_i}||\cdot ||w(t')-w(t'_i)||+||s_{t'_i}w(t'_i)-I||\leq 2\delta$ for $t'_i\leq t'\leq t'_{i+1}$. Then $||s(t')w(t')-I||\leq 4\delta$. Similarly, $||w(t')s(t')-I||\leq 4\delta$.\par
    Obviously, $||s(t')||\leq 2N$ and $\mathrm{prop}(s(t'))<r$. Thus $s(t')$ is a continuous homotopy between $v_1$ and $v_2$ in $QU_{4\delta,2N,r,k}(X)$.
\end{proof}

The following two lemmas can be viewed as the controlled version of the classical result in $K$-theory that stable homotopy equivalence of idempotents is the same as stably similarity.

\begin{lemma}\label{Lemma:HomotopyImpliesSimilar}
Let $0<\delta<1/100$. If $e$ is $(\delta,N,r)$-equivalent to $f$ in $QP_{\delta,N,r,k}(X)$, then there exist a positive number $m$ and an element $u$ in $QU_{\delta,C_1(N),C_2(N,\delta)r,k}(X)$ with $(\delta,C_1(N),C_2(N,\delta)r)$-inverse $v$, such that
                          $$||f\oplus I_m\oplus 0_m - v(e\oplus I_m\oplus 0_m)u||<C_3(N)\delta,$$
where $C_1(N)$ and $C_3(N)$ depend only on $N$, $C_2(N,\delta)$ depends only on $N$ and $\delta$.
\end{lemma}

\begin{proof}
By Lemma \ref{Lemma:HomotopyImpliesLipschitz}, there exists $\alpha_N>0,m\in \mathbb{N}$ such that $e\oplus I_m\oplus 0_m$ is $(2\delta,3N,r)$-equivalent to $f\oplus I_m\oplus 0_m$ by an $\alpha_N$-Lipschitz homotopy $e_{t'}$, i.e., $||e_{t'}-e_{t''}||\leq \alpha_N|t'-t''|$ for any $t',t''\in [0,1]$. There exists a partition $0=t'_0<t'_1<\ldots<t'_n=1$ such that
                                 $$\alpha_N|t'_{i+1}-t'_i|<\frac{1}{2N+1}$$
Let $w_i=((2e_{t'_i}-I)(2e_{t'_{i+1}}-I)+I)/2$. We have $I-w_i=(2e_{t'_i}-I)(e_{t'_i}-e_{t'_{i+1}})+2(e_{t'_i}-e^2_{t'_i})$. Then                                 $$||I-w_i||<||2e_{t'_i}-I||\cdot ||e_{t'_i}-e_{t'_{i+1}}||+2||e_{t'_i}-e^2_{t'_i}||<1/2+4\delta<1$$
Thus $w_i$ is an invertible element and $w^{-1}_i=\Sigma^{\infty}_{j=0}(1-w_i)^j$. Let $v_i=\Sigma^l_{j=0}(I-w_i)^j$ satisfying $||v_i-w^{-1}_i||<\delta/2((\max_{i}\{||w_i||,||w^{-1}_i||\}+1)^n)$. Let
             $$u=w_0w_1\ldots w_{n-1},v=v_{n-1}v_{n-2}\ldots v_0.$$
Then $\max\{||u||,||v||\} \leq C_1(N)$, $\max\{\mathrm{prop}(u(t)),\mathrm{prop}(v(t))\} \leq C_2(N,\delta)r$ for $t\in[0,1]^k$ and $\max\{||I-uv||,||I-vu||\}<\delta$, where $C_1(N)$ depends only on $N$ and $C_2(N,\delta)$ depends only on $N,\delta$.\par
By computation, we have $||e_{t'_i}w_i-w_ie_{t'_{i+1}}||<26N\delta$. Then $||ue_1-e_0u||<C'N$, where $C'$ depends only on $N$. Thus
           $$||e_1-v(e_0)u||=||e_1-vue_1+v(ue_1-e_0u)||<C_3(N)\delta,$$
where $C_3(N)$ depends only on $N$.
\end{proof}

\begin{lemma}\label{Lemma:SimilarImplyHomotopy}
    Let $N\geq 1$, $0<\delta<1/(800N^4), 0<\varepsilon<1/400$. For $e$ and $f$ in $QP_{\delta,N,r,k}(B^p_{L,0}(X)^+\otimes M_n(\mathbb{C}))$, if there exists $u$ in $QU_{\delta,N,r,k}(X)$ with $(\delta,N,r)$-inverse $v$ satisfying $||uev-f||<\varepsilon$, then $e\oplus 0_n$ is $(2\varepsilon+4N^4\delta,2N^3,3r)$-equivalent to $f\oplus 0_n$ in $QP_{2\varepsilon+4N^4\delta,2N^3,3r,k}(X)$.
\end{lemma}

\begin{proof}
Let $e_{t'}$ be a homotopy connecting $f\oplus 0_n$ to $e\oplus 0_n$ obtained by combining the linear homotopy connecting $f\oplus 0_n$ to $uev\oplus 0_n$ with the following homotopy connecting $uev\oplus 0_n$ to $e\oplus 0_n$:
          $$R(t')(u\oplus I_n)R^*(t')(e\oplus 0_n)R(t')(v\oplus I_n)R^*(t')$$
where $R(t')=\begin{pmatrix}
        \cos(\pi t'/2) & \sin(\pi t'/2)\\
        -\sin(\pi t'/2) & \cos(\pi t'/2)
        \end{pmatrix}$.
It is not difficult to verify $e_{t'}$ is a $(2\varepsilon+4N^4\delta,2N^3,3r)$-homotopy between $e$ and $f$.
\end{proof}

\begin{definition}
Let $X$ be a proper metric space, define
       $$GQP_{\delta,N,r,k}(X)=\{e-f:e,f\in QP_{\delta,N,r,k}(X), \pi(e)=\pi(f)\}$$

We say that $e_1-f_1$ is  $(\delta,N,r)$-equivalent to $e_2-f_2$ if $e_1\oplus f_2\oplus I_n\oplus 0_n$ is $(\delta,N,r)$-equivalent to $f_1\oplus e_2\oplus I_n\oplus 0_n$ for some $n$. This defines an equivalent relation on $GQP_{\delta,N,r,k}$.
\end{definition}

    For any $u\in QU_{\delta,N,r,k}(X)$ with a $(\delta,N,r)$-inverse $v$, let $Z_t(u)$ be a homotopy connecting $I\oplus I$ to $u\oplus v$ obtained by combining the linear homotopy connecting $I\oplus I$ to $uv\oplus I$ with the homotopy $(u\oplus I)R(t)(v\oplus I)R^*(t)$ connecting $uv\oplus I$ to $u\oplus v$, let $Z'_t(u)$ be a homotopy connecting $I\oplus I$ to $v\oplus u$ obtained by combining the linear homotopy connecting $I\oplus I$ to $uv\oplus I$ with the homotopy $R(t)(u\oplus I)R^*(t)(v\oplus I)$ connecting $uv\oplus I$ to $v\oplus u$, where
               $$R(t)=\begin{pmatrix}
                       \cos(\pi t/2) & \sin(\pi t/2)\\
                       -\sin(\pi t/2) &\cos(\pi t/2)
                      \end{pmatrix}.$$
    Let
               $$e_t(u)=Z_t(u)(I\oplus 0)Z'_t(u).$$
We have that
\begin{enumerate}
    \item $||e^2_t(u)-e_t(u)||<8N^6\delta$;
     \item $||e_t(u)||\leq 4N^4$ and $||I-e_t(u)||\leq 5N^4$;
     \item $\mathrm{prop}(e_t(u)(t'))\leq 2r$ for $t'\in [0,1]^k$.
\end{enumerate}
    Then we can define a map $\theta$ from $QU_{\delta,N,r,k}(X)$ to $GQP_{8N^6\delta,5N^4,2r,k+1}(X)$ by:
                     $$\theta(u)=e_t(u)-(I\oplus 0)$$\par
    It is not difficult to see that the definition of $\theta$ does not depend on the choice of $(\delta,N,r)$-inverse $v$ of $u$ in the sense of equivalence.\par

    The following result can be considered as a controlled version of a classical result in operator $K$-theory $K_1(A)\cong K_0(SA)$.
\begin{lemma}\label{Lemma:ControlledSuspension}
    $\theta:QU_{\delta,N,r,k}(X)\rightarrow GQP_{8N^6\delta,5N^4,2r,k+1}(X)$ is an asymptotic isomorphism in the following sense: \par
(1) For any $0<\delta<1/100,r>0,N\geq 1$, there exists $0<\delta_1<\delta, N_1\geq N$ and $0<r_1<r$, such that if two elements $u_1$ and $u_2$ in $QU_{\delta_1,N,r_1,k}(X)$ are $(\delta_1,N,r_1)$-equivalent, then $\theta(u_1)$ and $\theta(u_2)$ are $(\delta,N_1,r)$-equivalent, where $\delta_1$ depends only on $\delta$ and $N$; $N_1$ depends only on $N$ and $r_1$ depends only on $r$. \par
(2) For any $0<\delta<1/100,r>0,N\geq 1$, there exists $0<\delta_2<\delta, N_2\geq N$ and $0<r_2<r$, such that if $u'$ and $u''$ in $QU_{\delta_2,N,r_2,k}(X)$ satisfy $\theta(u')$ is $(\delta_2,N,r_2)$-equivalent to $\theta(u'')$, then $u'\oplus I_m$ is $(\delta,N_2,r)$-equivalent to $u''\oplus I_m$ for some $m\in \mathbb{N}$, where $\delta_2$ depends only on $\delta$ and $N$; $N_2$ depends only on $N$ and $r_2$ depends only on $r,\delta,N$.     \par
(3) For any $0<\delta<1/100,r>0,N\geq 1$, there exists $0<\delta_3<\delta, N_3\geq N$ and $0<r_3<r$, such that for each $e-e_m\in GQP_{\delta_3,N,r_3,k+1}(X)$, there exists $u\in QU_{\delta,N_3,r,k}(X)$ for which $\theta(u)$ is $(\delta,N_3,r)$-equivalent to $e-e_m$, where $\delta_3$ depends only on $\delta$ and $N$; $N_3$ depends only on $N$ and $r_3$ depends only on $r,\delta,N$.
\end{lemma}

\begin{proof}
(1) Let $v_i$ be the $(\delta_1,N,r_1)$-inverse of $u_i$ for $i=1,2$, $w(t)$ be the $(\delta_1,N,r_1)$-homotopy between $u_1$ and $u_2$. By Lemma \ref{Lemma:EquivImplyInverseEquiv}, there exists a $(4\delta_1,2N,r_1)$-homotopy $s(t)$ connecting $v_1$ and $v_2$ such that $||I-s(t)w(t)||$ and $||I-w(t)s(t)||$ are less than $4\delta_1$. Let $a(t)$ be a homotopy connecting $I$ to $v_2u_1$ obtained by combining the linear homotopy connecting $I$ to $v_1u_1$ with the homotopy $s(t)u_1$; let $a'(t)$ be a homotopy connecting $I$ to $v_1u_2$ obtained by combining the linear homotopy connecting $I$ to $v_1u_1$ with the homotopy $v_1w(t)$; let $b(t)$ be a homotopy connecting $I$ to $u_2v_1$ obtained by combining the linear homotopy connecting $I$ to $u_1v_1$ with the homotopy $w(t)v_1$; let $b'(t)$ be a homotopy connecting $I$ to $u_1v_2$ obtained by combining the linear homotopy connecting $I$ to $u_1v_1$ with the homotopy $u_1s(t)$. Define
                   $$x_t=Z_t(u_2)(a(t)\oplus b(t))Z'_t(u_1)$$
                   $$x'_t=Z_t(u_1)(a'(t)\oplus b'(t))Z'_t(u_2)$$
 We have that
 \begin{enumerate}[(i)]
  \item $\max\{||x_t||,||x'_t||\}\leq 8N^6$;
  \item $\max\{||I-x_t x'_t||,||I-x'_t x_t||\} < 64N^{10}\delta_1$;
  \item $\max\{\mathrm{prop}(x_t),\mathrm{prop}(x'_t)\}<6r_1$;
  \item  $\max\{||x_i-I||,||x'_i-I||<3\delta_1\}$ for $i=0,1$.
 \end{enumerate}
 Thus $x_t,x'_t\in QU_{64N^{10}\delta_1,8N^6,6r_1,k+1}(X)$. And
                   $$||x_te_t(u_1)x'_t-e_t(u_2)||<(184N^{14})\delta_1$$
By Lemma \ref{Lemma:SimilarImplyHomotopy}, we can select appropriate $\delta_1,N_1$ and $r_1$ satisfying (1). \par
(2) Let $v',v''$ be $(\delta_2,N,r_2)$-inverses of $u',u''$ respectively. By Lemma \ref{Lemma:HomotopyImpliesSimilar}, there exists an element $u$ in $QU_{\delta_2,C_1(N),C_2(N,\delta_2)r_2,k+1}(X)$ with inverse $v$, such that
                          $$||ue_t(u'\oplus I)v-e_t(u''\oplus I)||<C_3(N)\delta_2$$
i.e.
 $$||u_tZ_t(u'\oplus I)(I\oplus 0)Z'_t(u'\oplus I)v_t-Z_t(u''\oplus I)(I\oplus 0)Z'_t(u''\oplus I)||<C_3(N)\delta_2 \eqno(A)$$
where $t\in[0,1]$. Thus we have
 $$||Z'_t(u''\oplus I)u_tZ_t(u'\oplus I)(I\oplus 0)-(I\oplus 0)Z'_t(u''\oplus I)u_tZ_t(u'\oplus I)||<C_4(N)\delta_2 \eqno(B)$$
Let
$$Z'_t(u''\oplus I)u_tZ_t(u'\oplus I)=\begin{pmatrix}
                                        b_t & g_t\\
                                        h_t & d_t
                                        \end{pmatrix}.$$
Then by $(B)$, we obtain
                      $$||g_t||<C_4(N)\delta_2, ||h_t||<C_4(N)\delta_2. \eqno(*)$$
By $(A)$, we also have
$$||(I\oplus 0)Z'_t(u'\oplus I)v_tZ_t(u''\oplus I)-Z'_t(u'\oplus I)v_tZ_t(u''\oplus I)(I\oplus 0)||<C_5(N)\delta_2\eqno(C)$$
Let
$$Z'_t(u'\oplus I)v_tZ_t(u''\oplus I)=\begin{pmatrix}
                                        b'_t & g'_t\\
                                        h'_t & d'_t
                                        \end{pmatrix}.$$
Then by $(C)$, we obtain
                      $$||g'_t||<C_5(N)\delta_2, ||h'_t||<C_5(N)\delta_2. \eqno(**)$$
Thus by $(*)$ and $(**)$, we know that $b_t\in QU_{C_6(N)\delta_2,C_7(N),C_8(N,\delta_2)r_2,k+1}(X)$ with a $(C_6(N)\delta_2,C_7(N),C_8(N,\delta_2)r_2)$-inverse $b'_t$ such that
                                $$||c_0-I||\leq ||u_0-I||<\delta_2$$
                                $$||c_1-(v''\oplus I)(u'\oplus I)||<C_9(N)\delta_2.$$
Thus we can select appropriate $\delta_2,N_2$ and $r_2$ satisfying (2). \par
(3) $e(t)$ can be considered as a homotopy in $QP_{\delta_3,N,r_3,k}(X)$, where $t\in [0,1]$. We can assume $e(0)=e(1)=e_m=I\oplus 0$. By the proof of Lemma \ref{Lemma:HomotopyImpliesSimilar}, there exists a homotopy $w(t)$ in $QU_{\delta_3,C_1(N),C_2(N,\delta_3)r_3,k}(X)$ with inverse $s(t)$ for which $w(0)=s(0)=I$ such that
                          $$||w(t)(I\oplus 0\oplus I_m\oplus 0_m)s(t)-e(t)\oplus I_m\oplus 0_m||<C_3(N)\delta_3$$
for some $m\in \mathbb{N}$ and all $t\in [0,1]$. By some minor modifications of $w(t)$ and $s(t)$, we have
                         $$||w(1)(I\oplus 0)-(I\oplus 0)w(1)||<C_4(N)\delta_3.  \eqno(A)$$
Let
                               $$w(1)=\begin{pmatrix}
                                        u & g\\
                                        h & u'
                                        \end{pmatrix},
                                 s(1)=\begin{pmatrix}
                                        v & g'\\
                                        h' & v'
                                        \end{pmatrix}.
                                        $$
then by $(A)$, we obtain
                         $$\max\{||g||,||h||,||g'||,||h'||\}<C_4(N)\delta_3.$$
Thus $u$ and $u'$ are two elements in $QU_{C_5(N)\delta_3,C_6(N),C_7(N,\delta_3)r_3,k}(X)$ with inverse $v$ and $v'$ respectively. \par
Let $a_t$ be a homotopy connecting $I\oplus I\oplus I$ to $v'v\oplus I\oplus I$ obtained by combining the linear homotopy connecting $I\oplus I\oplus I$ to $v'u'\oplus I\oplus I$ with the rotation homotopy connecting $(v'\oplus I\oplus I)(u'\oplus I\oplus I)$ to $(v'\oplus I\oplus I)(v\oplus u\oplus u')$ with the homotopy $(v'\oplus I\oplus I)(v\oplus w(1-t))$ connecting $(v'\oplus I\oplus I)(v\oplus u\oplus u')$ to $v'v\oplus I\oplus I$. Similarly, let $b_t$ be a homotopy connecting $I\oplus I\oplus I$ to $uu'\oplus I\oplus I$. Define
     $$y_t=(w(t)\oplus I\oplus I)(I\oplus a(t))(Z'_t(u)\oplus I\oplus I)$$
     $$y'_t=(Z_t(u)\oplus I\oplus I)(I\oplus b(t))(s(t)\oplus I\oplus I)$$
then we have
    $$y_0=y'_0=I, \max\{||y_i-I||,||y'_i-I||\}<C_8(N)\delta_3 \eqno(A)$$
and
    $$||y_t(e_t(u)\oplus 0)y'_t-(e\oplus 0)||<C_9(N)\delta_3 \eqno(B)$$
Now by Lemma \ref{Lemma:SimilarImplyHomotopy}, we can choose appropriate $\delta_3,N_3,r_3$ on the basis of $(A)$ and $(B)$ satisfying (3). \par
Remark: we can also let
    $$y_t=(w(t)\oplus I\oplus I)(I\oplus Z'_t(u')s(t)\oplus I)(Z'_t(u)\oplus I\oplus I)$$
    $$y'_t=(Z_t(u)\oplus I\oplus I)(I\oplus w(t)Z_t(u')\oplus I)(s(t)\oplus I\oplus I)$$
\end{proof}

\subsection{Strongly Lipschitz homotopy invariance}
\begin{definition}[Yu \cite{YuLocalization}]
Let $f,g:X\rightarrow Y$ be two proper Lipschitz maps, a continuous homotopy $F(t,x)(t\in[0,1])$ between $f$ and $g$ is called \textit{strongly Lipschitz} if:
\begin{enumerate}
\item $F(t,x)$ is a proper map from $X$ to $Y$ for each $t$;
\item there exists a constant $C$, such that $d(F(t,x),F(t,y))\leq Cd(x,y)$ for all $x,y\in X$ and $t\in [0,1]$, this $C$ is called Lipschitz constant of $F$;
 \item $F$ is equicontinuous in $t$, i.e. for any $\varepsilon>0$, there exists $\delta>0$ such that $d(F(t_1,x),F(t_2,x))<\varepsilon$ for all $x\in X$ if $|t_1-t_2|<\delta$;
\item  $F(0,x)=f(x),F(1,x)=g(x)$ for all $x\in X$.
\end{enumerate}
\end{definition}

    $X$ is said to be \textit{strongly Lipschitz homotopy equivalent} to $Y$, if there exist proper Lipschitz maps $f:X\rightarrow Y$ and $g:Y\rightarrow X$ such that $fg$ and $gf$ are strongly Lipschitz homotopic to $\mathrm{id}_Y$ and $\mathrm{id}_X$, respectively.

\begin{lemma}\label{Lemma:Homotopy}
Let $f$ and $g$ be two Lipschitz maps from $X$ to $Y$. Let $F(t,x)$ be a strongly Lipschitz homotopy connecting $f$ to $g$ with Lipschitz constant $C$.  There exists $C_0>0$, such that for any $u\in QU_{\delta,N,r,k}(X)$, there exists a homotopy $w(t')$ in $QU_{D(N)\delta,N^{100},C_0r,k}(Y)$ for which
                       $$w(0)=\mathrm{Ad}((V_f,V^+_f))(u)\oplus I$$
                       $$w(1)=\mathrm{Ad}((V_g,V^+_g))(u)\oplus I$$
where $D(N)$ depends only on $N$ and $C_0$ depends only on $C$.
\end{lemma}
\begin{proof}
Choose $\{t_{i,j}\}_{i\geq 0,j\geq 0} \subseteq [0,1]$ satisfying \par
(1) $t_{0,j}=0, t_{i,j+1}\leq t_{i,j}, t_{i+1,j}\geq t_{i,j}$;\par
(2) there exists $N_j\rightarrow \infty$ such that $t_{i,j}=1$ for all $i\geq N_j$ and $N_{j+1}\geq N_j$; \par
(3) $d(F(t_{i+1,j},x),F(t_{i,j},x))<\varepsilon_j=r/(j+1), d(F(t_{i,j+1},x),F(t_{i,j},x))<\varepsilon_j$ for all $x\in X$.\par
Let $f_{i,j}(x)=F(t_{i,j},x)$. By Lemma \ref{Lemma:CoveringIsometry}, there exist an isometric operator $V_{f_{i,j}}:E_X^p\rightarrow E_Y^p$ and a contractive operator $V^+_{f_{i,j}}:E_Y^p\rightarrow E_X^p$ with $V^+_{f_{i,j}}V_{f_{i,j}}=I$ such that
  \begin{center}
  $\mathrm{supp}(V_{f_{i,j}}) \subseteq \{(x,y)\in X\times Y:d(f_{i,j}(x),y)<r/(1+i+j)\}$;\\
  $\mathrm{supp}(V^+_{f_{i,j}}) \subseteq\{(y,x)\in Y\times X:d(f_{i,j}(x),y)<r/(1+i+j)\}$.
  \end{center}
For each $i>0$, define a family of operators $V_i(t)(t\in [0,\infty))$ from $E^p_X\oplus E^p_X$ to $E^p_Y\oplus E^p_Y$ and a family of operators $V^+_i(t)(t\in [0,\infty))$ from $E^p_Y\oplus E^p_Y$ to $E^p_X\oplus E^p_X$ by
    $$V_i(t)=R(t-j)(V_{f_{i,j}}\oplus V_{f_{i,j+1}})R^*(t-j), t\in[j,j+1]$$
    $$V^+_i(t)=R(t-j)(V^+_{f_{i,j}}\oplus V^+_{f_{i,j+1}})R^*(t-j), t\in[j,j+1]$$
where
    $$R(t)=\begin{pmatrix}
           \cos(\pi t/2)  & \sin(\pi t/2)\\
           -\sin(\pi t/2) & \cos(\pi t/2)
           \end{pmatrix}.$$
Consider:
    $$u_0(t)=\mathrm{Ad}((V_f,V^+_f))(u)=V_f(t)(u(t)\oplus I)V^+_f(t)+(I-V_f(t)V^+_f(t));$$
    $$u_{\infty}(t)=\mathrm{Ad}((V_g,V^+_g))(u)=V_g(t)(u(t)\oplus I)V^+_g(t)+(I-V_g(t)V^+_g(t));$$
    $$u_i(t)=\mathrm{Ad}((V_i,V^+_i))(u)=V_i(t)(u(t)\oplus I)V^+_i(t)+(I-V_i(t)V^+_i(t)).$$
Let $v$ be the $(\delta,N,r)$-inverse of $u$. Similarly, we can define
                           $$u'_i(t)=\mathrm{Ad}((V_i,V^+_i))(v).$$
    For each $i$, define $n_i$ by
                       $$n_i=\left\{
                               \begin{array}{lr}
                               \max\{j:i\geq N_j\}, &  \{j:i\geq N_j\}\not= \emptyset;\\
                               0,                  &    \{j:i\geq N_j\}= \emptyset.
                               \end{array}
                               \right. $$
We can choose $V_{f_{i,j}}$ in such a way that: $u_i(t)=u_{\infty}$ where $t\leq n_i$. \par
    Define
          $$w_i(t)=\left\{
                     \begin{array}{ll}
                     u_i(t)(u'_{\infty}(t)),                 &  t\geq n_i;\\
                     (n_i-t)I+(t-n_i+1)u_i(t)u'_{\infty}(t), &  n_i-1\leq t \leq n_i;\\
                     I,                                       &  0\leq t \leq n_i-1.
                     \end{array}
                     \right. $$
Consider:
    \begin{align*}
    a&=\bigoplus^{\infty}_{i=0}(w_i\oplus I);\\
    b&=\bigoplus^{\infty}_{i=0}(w_{i+1}\oplus I);\\
    c&=(I\oplus I)\bigoplus^{\infty}_{i=1}(w_i\oplus I).
    \end{align*}
By the construction of $\{t_{i,j}\}$, we know that $a,b,c\in QU_{D_1(N)\delta,N^2,C_1r,k}(Y)$ for some constant $C_1$ depending only on $C$.
Let
    $$V_{i,i+1}(t')=R(t')(V_i\oplus V_{i+1})R^*(t'),t'\in [0,1];$$
    $$V^+_{i,i+1}(t')=R(t')(V^+_i\oplus V^+_{i+1})R^*(t'), t'\in [0,1].$$
Define
    $$u_{i,i+1}(t')=V_{i,i+1}(t')((u\oplus I)\oplus I)V^+_{i,i+1}(t')+(I-V_{i,i+1}(t')V^+_{i,i+1}(t')),$$
then
\begin{align*}
    u_{i,i+1}(0)&=(V_i(u\oplus I)V^+_i+(I-V_iV^+_i))\oplus I;\\
    u_{i,i+1}(1)&=(V_{i+1}(u\oplus I)V^+_{i+1}+(I-V_{i+1}V^+_{i+1}))\oplus I
\end{align*}
Using $u_{i,i+1}(t')$, we can construct a homotopy $s_1(t')$ in $QU_{D_2(N)\delta,N^{100},C_2r,k}(Y)$ for some $C_2\geq C_1$ depending only on $C$, such that
            $$s_1(0)=a, s_1(1)=b.$$ \par
    We can also construct a homotopy $s_2(t')$ in $QU_{D_3(N)\delta,N^{100},C_3r,k}(Y)$ for some $C_3\geq C_1$ depending only on $C$, such that
            $$s_2(0)=b\oplus I, s_2(1)=c\oplus I.$$ \par
    Finally, we define $w(t')$ to be the homotopy obtained by combining the following homotopies:\par
    (1) the linear homotopy between $(u_0\oplus I)\bigoplus^{\infty}_{i=1}(I\oplus I)$ and $c'a((u_{\infty}\oplus I)\bigoplus^{\infty}_{i=1}(I\oplus I))$;\par
    (2) $s'_2(1-t')a((u_{\infty}\oplus I)\bigoplus^{\infty}_{i=1}(I\oplus I)); $\par
    (3) $s'_1(1-t')a((u_{\infty}\oplus I)\bigoplus^{\infty}_{i=1}(I\oplus I)); $\par
    (4) the linear homotopy between $a'a((u_{\infty}\oplus I)\bigoplus^{\infty}_{i=1}(I\oplus I))$ and $(u_{\infty}\oplus I)\bigoplus^{\infty}_{i=1}(I\oplus I)$,\\
where $a',b',c',s'_1,s'_2$ are the $(D(N)\delta,N^{100},C_0r)$-inverses of $a,b,c,s_1,s_2$ respectively in $QU_{D_4(N)\delta,N^{100},C_4r,k}(Y)$ for some $C_4\geq \max\{C_1,C_2,C_3\}$ depends only on $C$.\par
    Therefore, $w(t')$ is the homotopy connecting $\mathrm{Ad}((V_f,V^+_f))(u)\bigoplus^{\infty}_{i=1}I$ to $\mathrm{Ad}((V_g,V^+_g))(u)\bigoplus^{\infty}_{i=1}I$.
\end{proof}

    By Lemma \ref{Lemma:ControlledSuspension}, we have the following result:
\begin{lemma}
Let $X,Y,f$ and $g$ be as in Lemma \ref{Lemma:Homotopy}. For any $0<\delta<1/100, N\geq 1,r\geq 0$, there exist $0<\delta_1<\delta, N_1\geq N, 0\leq r_1<r$ such that for any $e\in QP_{\delta_1,N,r_1,k}(X)(k>1)$, there exists a homotopy $e(t')(t'\in[0,1])$ in $QP_{\delta,N_1,r,k}(Y)$ satisfying
             $$e(0)=\mathrm{Ad}((V_f,V^+_f))(e\oplus 0)\oplus (I\oplus 0)$$
             $$e(1)=\mathrm{Ad}((V_g,V^+_g))(e\oplus 0)\oplus (I\oplus 0)$$
where $\delta_1$ depends only on $\delta$ and $N$; $N_1$ depends only on $N$ and $r_1$ depends only on $r,\delta,N,C$.
\end{lemma}

\subsection{Controlled cutting and pasting}\label{section:cutpaste}
\begin{definition}[Yu \cite{YuFAD}]\label{Def:ControlledBoundary}
Let $X$ be a proper metric space, $X_1$ and $X_2$ be two subspaces. The triple $(X;X_1,X_2)$ is said to satisfy the \textit{strong excision} condition if : \par
(1) $X=X_1\cup X_2$, $X_i$ is a Borel subset and $\mathrm{int}(X_i)$ is dense in $X_i$ for $i=1,2$;\par
(2) there exists $r_0>0,C_0>0$ such that (i) for any $r'\leq r_0$, $\mathrm{bd}_{r'}(X_1)\cap \mathrm{bd}_{r'}(X_2)=\mathrm{bd}_{r'}(X_1\cap X_2)$; (ii) for each $X'=X_1,X_2,X_1\cap X_2$ and any $r'\leq r_0$, $\mathrm{bd}_{r'}(X')$ is strongly Lipschitz homotopy equivalent to $X'$ with $C_0$ as the Lipschitz constant.
\end{definition}

    Let the triple $(X;X_1,X_2)$ be as above. Let $0<\delta<1/100$. For any $u\in QU_{\delta,N,r,k}(X)$ with $(\delta,N,r)$-inverse $v$, we take $uX_1=\chi_{X_1}u\chi_{X_1}$, the same for $vX_1$. Define
                            $$w_u=\begin{pmatrix}
                                    I & uX_1\\
                                    0 & I
                                  \end{pmatrix}
                                  \begin{pmatrix}
                                    I     & 0\\
                                    -vX_1 & I
                                  \end{pmatrix}
                                  \begin{pmatrix}
                                    I & uX_1\\
                                    0 & I
                                  \end{pmatrix}
                                  \begin{pmatrix}
                                    0 & -I\\
                                    I & 0
                                   \end{pmatrix}
                            $$
then
                            $$w^{-1}_u=\begin{pmatrix}
                                    0  & I\\
                                    -I & 0
                                  \end{pmatrix}
                                  \begin{pmatrix}
                                    I & -uX_1\\
                                    0 & I
                                  \end{pmatrix}
                                  \begin{pmatrix}
                                    I    & 0\\
                                    vX_1 & I
                                  \end{pmatrix}
                                  \begin{pmatrix}
                                    I & -uX_1\\
                                    0 & I
                                   \end{pmatrix}.
                            $$\par
    We define a homomorphism $\partial_0: QU_{\delta,N,r,k}(X)\rightarrow QP_{4N^4\delta,2N^6,6r,k}(\mathrm{bd}_{5r}(X_1) \cap \mathrm{bd}_{5r}(X_2))$ by
     $$\partial_0(u)=\chi_{\mathrm{bd}_{5r}(X_1) \cap \mathrm{bd}_{5r}(X_2)} w_u(I\oplus 0)w^{-1}_u \chi_{\mathrm{bd}_{5r}(X_1) \cap \mathrm{bd}_{5r}(X_2)}$$ \par
    Now we verify $\partial_0(u)\in QP_{4N^4\delta,2N^6,6r,k}(\mathrm{bd}_{5r}(X_1) \cap \mathrm{bd}_{5r}(X_2))$. Firstly, $||\partial_0(u)||$ and $||1-\partial_0(u)||$ are less than $2N^6$. Secondly, $\mathrm{prop}(\partial_0(u))<6r$. Finally, we estimate $||(\partial_0(u))^2-\partial_0(u)||$. For convenience, we take $Y=\mathrm{bd}_{5r}(X_1) \cap \mathrm{bd}_{5r}(X_2)$:
     $$||(\partial_0(u))^2-\partial_0(u)||=||\chi_Y w_u(I\oplus 0)w^{-1}_u\chi_{X_1-Y} w_u(I\oplus 0)w^{-1}_u\chi_Y||.$$
We now estimate $||\chi_Y w_u(I\oplus 0)w^{-1}_u\chi_{X_1-Y}||$. We have $\chi_{X_1}u\chi_{X_1-Y}=u\chi_{X_1-Y}$. Thus we can replace $uX_1$ by $u$ in $w_u(I\oplus 0)w^{-1}_u$. Then
           $$w_u(I\oplus 0)w^{-1}_u=\begin{pmatrix}
                                      (I-uv)uv+uv  &  (I-uv)u(I-vu)+u(I-vu)\\
                                      (I-vu)v      &  (I-vu)^2,
                                     \end{pmatrix}.
           $$
Thus
$$||\chi_Y w_u(I\oplus 0)w^{-1}_u\chi_{X_1-Y}||=||\chi_{Y\cap X_1}((w_u(I\oplus 0)w^{-1}_u)-(I\oplus 0))\chi_{X_1-Y}||<2N^2\delta.$$
Similarly,
$$||\chi_{X_1-Y} w_u(I\oplus 0)w^{-1}_u \chi_Y||<2N^2\delta.$$\par
    Assume that $r<r_0/5$, where $r_0$ is as in Definition \ref{Def:ControlledBoundary}. Let $f$ be the proper Lipschitz map from $\mathrm{bd}_{5r}(X_1) \cap \mathrm{bd}_{5r}(X_2)$ to $X_1\cap X_2$ realizing the strong Lipschitz homotopy equivalence in Definition \ref{Def:ControlledBoundary}. By Lemma \ref{Lemma:IsometryPairHom}, we have the pair $(V_f,V^+_f)$ corresponding to $\{\varepsilon_m\}$ for which $\sup_m(\varepsilon_m)<r/10$.\par
    We define the boundary map $\partial:QU_{\delta,N,r,k}(X)\rightarrow GQP_{4N^4\delta,2N^6,6C_0r,k}(X_1 \cap X_2)$ by
                     $$\partial(u)=\mathrm{Ad}((V_f,V^+_f))(\partial_0(u))-(I\oplus 0)$$ \par
    Then we consider the following sequence:
        $$QU_{\delta,N,r,k}(X_1) \oplus QU_{\delta,N,r,k}(X_2) \xrightarrow{j} QU_{\delta,N,r,k}(X) \xrightarrow{\partial} GQP_{4N^4\delta,2N^6,6C_0r,k}(X_1 \cap X_2)$$
where $j(u_1\oplus u_2)=(u_1+\chi_{X-X_1}) \oplus (u_2+\chi_{X-X_2}), r<r_0/5 $.

\begin{lemma}\label{Lemma:ControlledExact}
Let $(X;X_1,X_2)$ be as in Definition \ref{Def:ControlledBoundary} with $r_0,C_0$, then the above sequence is asymptotically exact in the following sense: \par
    (1) For any $0<\delta<1/100, N\geq 1, r>0$, there exist $0<\delta_1<\delta, N_1\geq N, 0<r_1<\min\{r,r_0/5 \}$, such that $\partial j(u_1\oplus u_2)$ is $(\delta,N_1,r)$-equivalent to $0$ for any $u_i\in QU_{\delta_1,N,r_1,k}(X_i)(i=1,2)$, where $\delta_1$ depends only on $\delta,N$; $N_1$ depends only on $N$ and $r_1$ depends only on $\delta,N,r$. \par
    (2) For any $0<\delta<1/100,N\geq 1,r>0$, there exists $0<\delta_2<\delta, N_2\geq N, 0<r_2<\min\{r,r_0/5 \}$, such that if $u$ is an element in $QU_{\delta_2,N,r_2,k}(X)$ for which $\partial(u)$ is $(\delta_2,N,r_2)$-equivalent to $0$ in $GQP_{\delta_2,N,r_2,k}(X)$, then there exist $u_i\in QU_{\delta,N_2,r,k}(X_i)(i=1,2)$ such that $j(u_1\oplus u_2)$ is $(\delta,N_2,r)$-equivalent to $u$, where $\delta_2$ depends only on $\delta,N$; $N_2$ depends only on $N$ and $r_2$ depends only on $\delta,N,r,r_0,C_0$.
\end{lemma}
\begin{proof}
(1) follows from the definition of the boundary map and Lemma \ref{Lemma:SimilarImplyHomotopy}.\par
(2) By strong homotopy invariance of $QP$, for any $0<\delta'<\delta, N\geq 1, 0<r'_2<\min\{r,r_0/5 \}$, there exist $\delta_2<\delta', N'>N, 0<r_2<r'_2$ ($\delta_2$ depends only on $\delta',N$; $N'$ depends only on $N$; $r_2$ depends only on $r'_2,\delta',N,r_0,C_0$), such that , for any $u\in QU_{\delta_2,N,r_2,k}(X)$ whose boundary $\partial(u)$ is $(\delta_2,N,r_2)$-equivalent to $0$, then $\partial_0(u)$ is $(\delta',N',r'_2)$-equivalent to $0$. By Lemma \ref{Lemma:HomotopyImpliesSimilar} there exists an element $y$ in $QU_{\delta',C_1(N'),C_2(N',\delta')r'_2,k}(\mathrm{bd}_{5r_2}(X_1) \cap \mathrm{bd}_{5r_2}(X_2))$ with $(\delta',C_1(N'),C_2(N',\delta')r'_2)$-inverse $y'$, such that
                          $$||xw(I\oplus 0)w^{-1}x'-(I\oplus 0)||<C_3(N')\delta',$$
where $x=y+\chi_{X-\mathrm{bd}_{5r_2}(X_1) \cap \mathrm{bd}_{5r_2}(X_2)}, x'=y'+\chi_{X-\mathrm{bd}_{5r_2}(X_1) \cap \mathrm{bd}_{5r_2}(X_2)}, w=w_{u\oplus I}$.\par
    This implies that
                          $$||xw(I\oplus 0)-(I\oplus 0)xw||<C_4(N')\delta'.$$
Thus we have
           $$xw=\begin{pmatrix}
                 a   &   b\\
                 c   &   d
                \end{pmatrix},
             ||b||\leq C_4(N')\delta',
             ||c||\leq C_4(N')\delta',                  \eqno(A)
           $$
           $$w^{-1}x'=\begin{pmatrix}
                 a'  &   b'\\
                 c'  &   d'
                \end{pmatrix},
             ||b'||\leq C_4(N')\delta',
             ||c'||\leq C_4(N')\delta',                  \eqno(B)
           $$
Define
           $$v_1=a\chi_{\mathrm{bd}_{5r_2}}(X_1), v'_1=\chi_{\mathrm{bd}_{5r_2}}(X_1)a',$$
$(A)$ and $(B)$ tell us that $v_1\in QU_{(C_4(N')+1)\delta',C_1(N')N^3_2,(C_2(N',\delta')+3)r'_2,k}(\mathrm{bd}_{5r_2}(X_1))$ with inverse $v'_1$. $(A)$ and $(B)$ together with the definition of $w$, implies that
           $$||\chi_{X-\mathrm{bd}_{10r_2}(X_2)}(v'_1(u\oplus I-I))||<C_5(N')\delta',                  \eqno(C)$$
           $$||(v'_1(u\oplus I-I))\chi_{X-\mathrm{bd}_{10r_2}(X_2)}||<C_5(N')\delta'.                  \eqno(D)$$
Define
           $$v_2=\chi_{\mathrm{bd}_{10r_2}(X_2)}(v'_1(u\oplus I))\chi_{\mathrm{bd}_{10r_2}(X_2)}, v'_2=\chi_{\mathrm{bd}_{10r_2}(X_2)}((u'\oplus I)v_1)\chi_{\mathrm{bd}_{10r_2}(X_2)},$$
where $u'$ is the $(\delta_2,N,r_2)$-inverse of $u$. \par
$(C)$ and $(D)$ tell us that $v_2\in QU_{C_6(N')\delta', C_7(N'),C_8(N',\delta')r'_2,k}(\mathrm{bd}_{10r_2}(X_2))$ with quasi inverse $v'_2$. \par
    We require $0<r_2<r_0/10$. Let $f_1$ be the proper strong Lipschitz map from $\mathrm{bd}_{5r}(X_1)$ to $X_1$ realizing the strong Lipschitz homotopy equivalence. Let $f_2$ be the proper strong Lipschitz map from $\mathrm{bd}_{10r}(X_2)$ to $X_2$ realizing the strong Lipschitz homotopy equivalence. Define $u_i=\mathrm{Ad}((V_{f_i},V^+_{f_i}))(v_i)$ for $i=1,2$, where the pair $(V_{f_i},V^+_{f_i})$ corresponds to $\{\varepsilon_k\}$ for which $\sup_k(\varepsilon_k)<r'_2$. \par
    By $(C)$ and $(D)$, we have that $(v_1+\chi_{X-\mathrm{bd}(5r_2)(X_1)})\oplus (v_2+\chi_{X-\mathrm{bd}(10r_2)(X_2)})$ is $(C_9(N')\delta',C_{10}(N'),C_{11}(N',\delta')r'_2)$-equivalent to $u\oplus I$.
     Note that $C_j(N')$ depends only on $N'$ for $j=1,3,4,5,6,7,9,10$, $C_j(N_2,\delta')$ depends only on $N',\delta',C_0$ for $j=2,8,11$. \par
    By Lemma \ref{Lemma:Homotopy}, we can choose appropriate $\delta',N_2$ and $r'_2$ such that $u_1$ and $u_2$ satisfy the desired properties of (2), where $\delta'$ depends only on $\delta,N$; $N_2$ depends only on $N$; $r'_2$ depends only on $r,\delta,N,r_0,C_0$.
\end{proof}

\begin{corollary}
By Lemma $\ref{Lemma:ControlledSuspension}$ and Lemma $\ref{Lemma:ControlledExact}$, we have the following asymptotically exact sequence for $QU$ when $k>1$:
  $$QU_{\delta,N,r,k}(X_1) \oplus QU_{\delta,N,r,k}(X_2) \rightarrow QU_{\delta,N,r,k}(X) \rightarrow QU_{\delta,N,r,k-1}(X_1 \cap X_2)$$
\end{corollary}

\section{Spaces with finite asymptotic dimension}

    In this section, we will recall some facts about spaces with finite asymptotic dimension, and verify the $L^p$ coarse Baum-Connes conjecture for spaces with finite asymptotic dimension.
\begin{definition}[Gromov \cite{GromovHyperbolic}]
    The \textit{asymptotic dimension} of a metric space $X$ is the smallest integer $m$ such that for any $r>0$, there exists a uniformly bounded cover $C=\{U_i\}_{i\in I}$ of $X$ for which the $r$-multiplicity of $C$ is at most $m+1$; i.e. no ball of radius $r$ in the metric space intersects more than $m+1$ members of $C$. If no such $m$ exists, we say $X$ has infinite asymptotic dimension.
\end{definition}   

     A finitely generated group can be viewed as a metric space with a left-invariant \textit{word-length metric}. To be more precise, for a group $\Gamma$ with a finite symmetric generating set $S$, for any $\gamma\in\Gamma$, we define its length
    $$l_S(\gamma):=\min\{n:\gamma=s_1 \ldots s_n, s_i\in S\}$$
the word-length metric $d_S$ on $\Gamma$ is defined by
    $$d_S(\gamma_1,\gamma_2):=l_S(\gamma_1^{-1}\gamma_2)$$
for all $\gamma_1,\gamma_2\in \Gamma$. We remark that for any two finite symmetric generating sets $S_1,S_2$ of $\Gamma$, $(\Gamma,d_{S_1})$ is quasi-isometric to $(\Gamma,d_{S_2})$.

\begin{remark}
  Now we give some facts about asymptotic dimension:
  \begin{enumerate}
  \item The concept of asymptotic dimension is a coarse geometric analogue of the covering dimension in topology;
  \item Hyperbolic groups have finite asymptotic dimension as a metric space with word-length metric \cite{GromovHyperbolic}\cite{RoeHyperbolic}; 
  \item The class of finitely generated groups with finite asymptotic dimension is hereditary (Proposition 6.2 in \cite{YuFAD}) , i.e., if a finitely generated group $\Gamma$ has finite asymptotic dimension as a metric space with word-length metric, then any finitely generated subgroup of $\Gamma$ also has finite asymptotic dimension as a metric space with word-length metric.  
  \item If $\Gamma$ is a discrete subgroup of an almost connected Lie group, e.g. $\mathrm{SL}(n,\Z)$, then $\Gamma$ has finite asymptotic dimension.
  \item CAT(0) cube complexes have finite asymptotic dimension. \cite{WrightCAT0FAD}
  \item Certain relative hyperbolic groups have finite asymptotic dimension. \cite{OsinRelativeHyperbolic}
  \item Certain Coxeter groups have finite asymptotic dimension. \cite{DranishnikovFAD}
  \item Mapping class groups have finite asymptotic dimension. \cite{MappingClassGroup}
  \end{enumerate}
\end{remark}

\begin{construction}\label{construct:anticech}
    Let $X$ be a proper metric space with asymptotic dimension $m$. By the definition of asymptotic dimension there exists a sequence of covers $C_k$ of $X$ for which there exists a sequence of positive numbers $R_k\rightarrow \infty$ such that
    \begin{enumerate}
    \item  $R_{k+1}>4R_k$ for all $k$;
    \item  diameter$(U)<R_k/4$ for all $U\in C_k$;
    \item the $R_k$-multiplicity of $C_{k+1}$ is at most $m+1$, i.e. no ball with radius $R_k$ intersects more than $m+1$ members of $C_k+1$.
    \end{enumerate}
    Let $C_k'=\{B(U,R_k):U\in C_{k+1}\}$, where $B(U,R_k)=\{x\in X:d(x,U)<R_k\}$. (1), (2) and (3) imply that $\{C_k'\}$ is an anti-\v Cech system for $X$.\par
    Fix a positive integer $n_0$. For each $n>n_0$, let $r_n=\frac{R_n}{2R_{n_0+1}}-4$. By property (1) of the sequence ${R_k}$, there exists $n_1>n_0$ such that $r_n>2$ if $n>n_1$ and there exists a sequence of nonnegative smooth functions $\{\chi_n\}_{n>n_1}$ on $[0,\infty)$ for which
    \begin{enumerate}
     \item $\chi_n(t)=1$ for all $0\leq t\leq 2$, and $\chi_n(t)=0$ for all $t\geq r_n$;
     \item there exists a sequence of positive numbers $\varepsilon_n\rightarrow 0$ satisfying $|\chi'_n(t)|<\varepsilon_n\leq 1$ for all $n>n_1$.
    \end{enumerate}
    For each $U\in C_{n+1}(n>n_1)$, define
        $$U'=\{V\in N_{C'_{n_0}}:V\in C'_{n_0},U \cap V \not= \emptyset\}$$ \par
    We define a map $G_n:N_{C'_{n_0}}\rightarrow N_{C'_n}$ by
        $$G_n(x)=\sum_{U\in C_{n+1}} \frac{\chi_n(d(x,U'))}{\sum_{V\in C_{n+1}} \chi_n(d(x,V'))} B(U,R_n)$$
for all $x\in N_{C'_{n_0}}$.\par
    Let $n>n_1$, we define a map $i_{n_0n}:N_{C'_{n_0}}\rightarrow N_{C'_n}$ in such a way that, for each $V\in C_{n_0+1}$,
        $$i_{n_0n}(B(V,R_{n_0}))=B(U,R_n)$$
for some $U\in C_{n+1}$ satisfying $U\cap V\not= \emptyset$.\par
    Let $F_t$ be the linear homotopy between $G_n$ and $i_{n_0n}$, i.e. $F_t(x)=tG_n(x)+(1-t)i_{n_0n}(x)$ for all $t\in [0,1]$ and $x\in N_{C'_{n_0}}$.\par
\end{construction}
    By the above construction, we have the following important lemma:
\begin{lemma} \label{Lemma:PropagationArbitrarySmall}(Lemma 6.3 in \cite{YuFAD})
Let $X$ be a proper metric space with finite asymptotic dimension $m$, and $G_n,F_t$ and $i_{n_0n}$ be as above, then
\begin{enumerate}
\item $G_n$ is a proper Lipschitz map with a Lipschitz constant depending only on $m$;
\item $F_t$ is a strong Lipschitz homotopy between $G_n$ and $i_{n_0n}$ with a Lipschitz constant depending only on $m$;
\item For any $\varepsilon>0,R>0$, there exists $K>0$ such that $d(G_n(x),G_n(y))<\varepsilon$ if $n>K,d(x,y)<R$.
\end{enumerate}
\end{lemma}


The following lemma plays a crucial role in the proof of Theorem \ref{Thm:LpBCCFAD}. Its proof is based on the Eilenberg swindle argument and the controlled cutting and pasting exact sequence in Section \ref{section:cutpaste}.
\begin{lemma}\label{Lemma:VanishingControlledObstruction}
Let $X$ be a simplicial complex with finite dimension $m$ and endowed with $\ell^1$ metric. For any $k>m+1, 0<\delta<1/100, N\geq 1, r>0$, there exist $0<\delta_1\leq \delta, N_1\geq N, 0<r_1<r$, such that every element $u$ in $QU_{\delta_1,N,r_1,k}(X)$ is $(\delta,N_1,r)$-equivalent to $I$, where $\delta_1$ depends only on $\delta,N$; $N_1$ depends only on $N$ and $r_1$ depends only on $r,\delta,N$.
\end{lemma}
\begin{proof}
Let $X^{(n)}$ be the $n$-skeleton of $X$, we will prove our lemma for $X^{(n)}$ by induction on $n$. \par
    When $n=0$, we choose $r_1=\min\{r,2\}$. Let $v$ be the $(\delta_1,N,r_1)$-inverse of $u$. Then $\mathrm{prop}(u(t))=\mathrm{prop}(v(t))=0$. For $t_0\in [0,\infty)$, we define:
             $$u_{t_0}(t)=\left \{
                                \begin{array}{ll}
                                I,         &  0\leq t\leq t_0;\\
                                u(t-t_0),  &  t_0\leq t<+\infty.
                                \end{array} \right.
             $$
Similarly, we can define $v_{t_0}$ for $t_0\in [0,\infty)$. Thus $v_{t_0}$ is the $(\delta_1,N,r_1)$-inverse of $u_{t_0}$. \\
Define
             $$E^{p,\infty}_X=(\oplus^{\infty}_{k=0} E^p_X)\oplus E^p_X$$
    Let $w_1(t')$ be the linear homotopy between $u\oplus^{\infty}_{k=1}I \oplus I$ and $u\oplus^{\infty}_{k=1}u_kv_k\oplus I$. Let $w_2(t')=(\oplus^{\infty}_{k=0}u_k\oplus I)(I\oplus^{\infty}_{k=1}v_{k-t'}\oplus I)$, where $t'\in [0,1]$. \par
    Let $T,T^*:E^{p,\infty}_X\rightarrow E^{p,\infty}_X$ be linear maps defined by
             \begin{align*}T((h_0,h_1,\ldots),h))&=(0,h_0,h_1,\ldots),h)\\
             T^*((h_0,h_1,\ldots),h))&=(h_1,h_2,\ldots),h)
             \end{align*}
Thus
             $$I\oplus^{\infty}_{k=1}v_{k-1}\oplus I=T(\oplus^{\infty}_{k=0}(v_{k}-I)\oplus 0)T^*+I.$$
Hence there exists a homotopy $s_1(t')(t'\in[0,1])$ connecting $I\oplus^{\infty}_{k=1}v_{k-1}\oplus I$ and $\oplus^{\infty}_{k=0}v_{k}\oplus I$.\par
    Let $s_2(t')(t'\in[0,1])$ be the linear homotopy between $\oplus^{\infty}_{k=0}u_kv_k\oplus I$ and $\oplus^{\infty}_{k=0}I\oplus I$.\par
    Define
             $$w(t')=\left\{
                          \begin{array}{ll}
                          w_1(4t'),                                        &   0\leq t'\leq 1/4;\\
                          w_2(4t'-1),                                      &   1/4\leq t'\leq 1/2;\\
                          (\oplus^{\infty}_{k=0}u_k\oplus I)s_1(4t'-2),    &   1/2\leq t'\leq 3/4;\\
                          s_2(4t'-3),                                      &   3/4\leq t'\leq 1.\\
                          \end{array}\right.
             $$ \par
    It is not difficult to see $w(t')$ is the homotopy connecting $u\oplus I$ to $I$, thus we can choose appropriate $\delta_1$ and $N_1$ satisfying the lemma. \par
    Assume by induction that the lemma holds for $n=m-1$, next we will prove the lemma holds for $n=m$. For each simplex $\triangle$ of dimension $m$ in $X$, we let
    $$\triangle_1=\{x\in \triangle: d(x,c(\triangle))\leq 1/100\},\triangle_2=\{x\in \triangle: d(x,c(\triangle))\geq 1/100\},$$
where $c(\triangle)$ is the center of $\triangle$. \par
    Let
      $$X_1=\bigcup_{\triangle:\text{simplex of dimension $m$ in X}} \triangle_1;$$
      $$X_2=\bigcup_{\triangle:\text{simplex of dimension $m$ in X}} \triangle_2.$$
Notice that: \\
(1) $X_1$ is strongly Lipschitz homotopy equivalent to
                  $$\{c(\triangle): \text{$\triangle$ is $m$-dimensional simplex in X}\};$$
(2) $X_2$ is strongly Lipschitz homotopy equivalent to $X^{(m-1)}$;\\
(3) $X^{(m)}=X_1 \cup X_2$ and $X_1 \cap X_2$ is the disjoint union of the boundaries of all $m$-dimensional $\triangle_1$ in $X^{(m)}$. \par
    (1) and (2) together with strongly Lipschitz homotopy invariance of $QU$ and the induction hypothesis, imply that our lemma holds for $X_1$ and $X_2$.\par
    By strongly Lipschitz homotopy invariance of $QU$ and the controlled cutting and pasting exact sequence, we also know that our lemma holds for $X_1\cap X_2$. \par
    Obviously, $(X^{(m)},X_1,X_2)$ satisfies the strong excision condition, thus we can complete our induction process by using the controlled cutting and pasting exact sequence and the controlled five lemma.
\end{proof}

Now we are ready to prove the main theorem of this section.
\begin{theorem}\label{Thm:LpBCCFAD}
For any $p\in [1,\infty)$, the $L^p$ coarse Baum-Connes conjecture holds for proper metric spaces with finite asymptotic dimension.
\end{theorem}
\begin{proof}
    Let $X$ be a proper metric space with asymptotic dimension $m$. By Theorem \ref{Thm:VanishingObstruction}, it is enough to prove that
               $$\lim_{n\rightarrow \infty}K_i(B^p_{L,0}(N_{C'_n}))=0,$$
where $C'_n$ is as in Construction \ref{construct:anticech}. \par
     Lemma \ref{Lemma:QuasiVSTrueProjection}, \ref{Lemma:QuasiVSTrueInvertible}  and \ref{Lemma:ControlledSuspension} tell us that any element $[q]$ in $K_i(B^p_{L,0}(N_{C'_{n_0}}))$ can be represented as an element $u$ in $QU_{\delta_1,N,r,k}(N_{C'_{n_0}})$ for some $N,r$ and $k>m+1$, where $\delta_1$ is as in Lemma \ref{Lemma:VanishingControlledObstruction} for some $0<\delta<1/100$. Let
               $$u_n=\mathrm{Ad}((V_{G_n},V^+_{G_n}))(u),$$
where $G_n$ is as in Lemma \ref{Lemma:PropagationArbitrarySmall}, $\mathrm{Ad}((V_{G_n},V^+_{G_n}))$ is defined by $\{\varepsilon_m\}$ for which $\sup(\varepsilon_m)<r_1/10$, where $r_1$ is as in Lemma \ref{Lemma:VanishingControlledObstruction}. \par
    By Lemma \ref{Lemma:PropagationArbitrarySmall} (3), there exists $K>0$ such that
               \begin{center}
               $\mathrm{prop}(u_n)<r_1,$ for $n>K.$
               \end{center}
     Since the asymptotic dimension of $X$ is $m$, we have $\dim(N_{C'_n})\leq m$ for all $n$. By Lemma $\ref{Lemma:VanishingControlledObstruction}$, we have that $u_n$ is $(\delta,N_1,r)$-equivalent to $I$ in $QU_{\delta,N_1,r,k}(N_{C'_n})$ for $n>K$.\par
    By Lemma $\ref{Lemma:PropagationArbitrarySmall}$ (2), strongly Lipschitz homotopy invariance of $QU$, Lemma \ref{Lemma:QuasiVSTrueProjection} (2) and Lemma \ref{Lemma:QuasiVSTrueInvertible} (2), we have that $\mathrm{Ad}((V_{i_{n_0n}},V^+_{i_{n_0n}}))(u)$ and $u_n$ correspond to the same element in $K_i(B^p_{L,0}(N_{C'_n}))$. \par
    Thus $[q]=0$ in $\lim_{n\rightarrow \infty}K_i(B^p_{L,0}(N_{C'_n}))$.
\end{proof}

\section{$K$-theory of $L^p$ Roe algebras}
    In this section, we shall use the dual $L^p$ $K$-homology as a bridge to prove that the $L^p$ $K$-homology is independent of $p$. Combining  the Theorem \ref{Thm:LpBCCFAD}, we obtain that the $K$-theory of the $L^p$ Roe algebra does not depend on $p\in (1,\infty)$ for spaces with finite asymptotic dimension.  \par

\subsection{Dual $L^p$ Localization algebra and dual $L^p$ $K$-homology}



    Let $p\in (1,\infty)$, $Z$ and $Z'$ be countable discrete measure spaces, then $\ell^p(Z)$ has a natural Schauder basis $\{e_i\}_{i\in Z}$, where $e_i(z)=1$ for $i=z$ and $e_i(z)=0$ for $i\not=z$. Similarly, $\ell^p(Z')$ has a natural Schauder basis $\{e'_i\}_{i\in Z'}$. Let $T$ be a bounded operator from $\ell^p(Z)$ to $\ell^p(Z')$, $T$ can be considered as a countably dimensional matrix under the Schauder basis $\{e_i\}$ and $\{e'_i\}$. We can define $T^*$ as the transpose of the matrix of $T$. We call $T$ be a \textit{dual-operator}, if $T^*$ is a bounded operator from $\ell^p(Z')$ to $\ell^p(Z)$ under the Schauder basis $\{e'_i\}$ and $\{e_i\}$. We call $T$ a \textit{compact dual-operator}, if $T$ and $T^*$ are compact operators from $\ell^p(Z)$ to $\ell^p(Z')$ and from $\ell^p(Z')$ to $\ell^p(Z)$, respectively. We define the  \textit{maximal norm} of dual-operator $T$ by $||T||_{\max}:=\max\{||T||,||T^*||\}$. \par
    For $p\in(1,\infty)$, let $\mathcal{B}^*(\ell^p(Z),\ell^p(Z'))$ be the Banach space of all dual-operators from $\ell^p(Z)$ to $\ell^p(Z')$ with maximal norm. Let $\mathcal{K}^*(\ell^p(Z),\ell^p(Z'))$ be the Banach space of all compact dual-operators from $\ell^p(Z)$ to $\ell^p(Z')$. It is easy to see that $\mathcal{K}^*(\ell^p(Z))$ is a closed ideal of $\mathcal{B}^*(\ell^p(Z))$.
\begin{remark}
    For $p\in (1,\infty)$, let $q$ be the dual number of $p$, i.e., $1/p+1/q=1$. If $T$ is a dual-operator acting on $\ell^p(Z)$, then $T$ can be considered as a bounded operator acting on $\ell^q(Z)$ and $||T||_{\ell^q(Z)}=||T^*||_{\ell^p(Z)}$. This is why we call such $T$ a dual-operator. Note that $\mathcal{B}^*(\ell^p(Z))=\mathcal{B}^*(\ell^q(Z))$ for $p,q\in (1,\infty)$ and $1/p+1/q=1$.
\end{remark}

\begin{lemma}\label{Lemma:LpCompact}
    Let $p\in (1,\infty)$, $Z$ be a countable discrete measure space. If we fixed a bijection between $Z$ and $\mathbb{N}$, then $\ell^p(Z)$ has a natural Schauder basis $\{e_i\}_{i\in \mathbb{N}}$. For any $K\in \mathcal{K}^*(\ell^p(Z))$, we have
                      $$\lim_{n\rightarrow \infty}F_nKF_n=K$$
in $\mathcal{K}^*(\ell^p(Z))$, where $F_n$ is the coordinate projection from $\ell^p(Z)$ to the subspace generated by $e_1,\cdots,e_n$.
\end{lemma}
\begin{proof}
    We just need to prove $\lim_{n\rightarrow \infty}||F_nKF_n-K||_{max}=0$, i.e. $$\lim_{n\rightarrow \infty}||F_nKF_n-K||_{l^p(Z)}=0 \text{ and }\lim_{n\rightarrow \infty}||F^*_nK^*F^*_n-K^*||_{\ell^p(Z)}=0.$$ These are true by the Proposition 1.8 in \cite{PhillipsLp}. 
\end{proof}
    This lemma is false for $p=1$, N.C. Phillips constructed a rank one operator without this property in \cite{PhillipsLp}.
\begin{corollary}\label{Cor:KtheoryLpCompact}
    Let $p\in (1,\infty)$, $Z$ be a countable discrete measure space, then $K_1(\mathcal{K}^*(\ell^p(Z)))=0$
                   and  $K_0(\mathcal{K}^*(\ell^p(Z)))=\mathbb{Z}$ generated by a rank one idempotent.
\end{corollary}
\begin{proof}
     Lemma \ref{Lemma:LpCompact} implies that $\mathcal{K}^*(\ell^p(Z))$ can be represented as the direct limit of matrix algebras. By the continuous property of $K$-group, we complete the proof.
\end{proof}

    Let $X$ be a proper metric space, $p\in (1,\infty)$. Recall that an $L^p$-$X$-module is an $L^p$-space $E_X^p=\ell^p(Z_X)\otimes \ell^p=\ell^p(Z_X,\ell^p)$ equipped with a natural point-wise multiplication action of $C_0(X)$ by restricting to $Z_X$, where $Z_X$ is a countable dense subset in $X$.
    This action can naturally induce a morphism from $C_0(X)$ to $\mathcal B^\ast(E^p_X)$.

\begin{definition}
    Let $X,Y$ be proper metric spaces, $T\in\mathcal{B}^*(E^p_X,E^p_Y)$. The \textit{support} of $T$, denoted $\mathrm{supp}(T)$, consists of all points $(x,y)\in X\times Y$ such that $\chi_VT\chi_U\not=0$ for all open neighbourhoods $U$ of $x$ and $V$ of $y$.
\end{definition}

    We remark that $\mathrm{supp}(T)$ has the same properties like Remark \ref{Remark:support}.

\begin{definition}
    Let $X$ be proper metric spaces, $T$ be an element in $\mathcal{B}^*(E^p_X)$.\par
(1) The \textit{propagation} of $T$, is defined to be $\mathrm{prop}(T):=\sup\{d(x,y):(x,y)\in \mathrm{supp}(T)\}$; \par
(2) $T$ is said to be \textit{locally compact dual-operator} if $\chi_K T$ and $T \chi_K$ are in $\mathcal{K}^*(E^p_X)$ for any compact subset $K$ in $X$.
\end{definition}

\begin{definition}
The \textit{dual $L^p$ Roe algebra} of $E_X^p$, denoted $B^{p,\ast}(E_X^p)$, is defined to be the maximal-norm closure of the algebra of all locally compact dual-operators acting on $E_X^p$ with finite propagations.
\end{definition}

    Let $X,Y$ be two proper metric spaces, and $f$ be a continuous coarse map from $X$ to $Y$.
      Let $V_f$ and $V^+_f$  be an isometric dual-operator and a contractive dual-operator, respectively, constructed in Lemma \ref{Lemma:CoveringIsometry}. Thus we have the following lemma.
\begin{lemma}
    Let $f$, $E_X^p$ and $E_Y^p$ be as above. Then the pair $(V_f,V_f^+)$ gives rise to a homomorphism $\mathrm{ad}((V_f,V_f^+)):B^{p,\ast}(E_X^p)\rightarrow B^{p,\ast}(E_Y^p)$ defined by:
                        $$\mathrm{ad}((V_f,V_f^+))(T)=V_fTV_f^+$$
for all $T\in B^{p,\ast}(E_X^p)$.\par
Moreover, the map $\mathrm{ad}((V_f,V_f^+))_*$ induced by $\mathrm{ad}((V_f,V_f^+))$ on $K$-theory depends only on $f$ and not on the choice of pair $(V_f,V_f^+)$.
\end{lemma}
\begin{proof}
    The proof of this lemma is same as the proof of Lemma \ref{Lemma:CoveringIsometryPair}.
\end{proof}

\begin{corollary}
    For different $L^p$-$X$-modules $E_X^p$ and $E_X'^{p}$, $B^{p,\ast}(E_X^p)$ is non-canonically isomorphic to $B^{p,\ast}(E_X'^{p})$, and $K_*(B^{p,\ast}(E_X^p))$ is canonically isomorphic to $K_*(B^{p,\ast}(E_X'^{p}))$.
\end{corollary}
For convenience, we replace $B^{p,\ast}(E_X^p)$ by $B^{p,\ast}(X)$ representing the dual $L^p$ Roe algebra of $X$.

\begin{definition}
   Let $X$ be a proper metric space. The \textit{dual $L^p$ localization algebra} of $X$, denoted $B^{p,\ast}_L(X)$, is defined to be the norm closure of the algebra of all bounded and uniformly norm-continuous functions $f$ from $[0,\infty)$ to $B^{p,\ast}(X)$ such that
  \begin{center}
   prop($f(t)$) is uniformly finite and prop($f(t)$)$\rightarrow 0$ as $t\rightarrow \infty$.
  \end{center}
  The \textit{propagation} of $f$ is defined to be $\sup\{\mathrm{prop}(f(t)):t\in [0,\infty)\}$.
\end{definition}

    We have the following lemma for  dual $L^p$ localization algebra just like Lemma \ref{Lemma:IsometryPairHom}.
\begin{lemma}
    Let $X,Y$ be two proper metric spaces, $f$ be a uniformly continuous coarse map from $X$ to $Y$ and $\{\varepsilon_k\}_k$ be a sequence of positive numbers such that $\varepsilon_k \rightarrow 0$ as $k \rightarrow \infty$, then the pair $(V_f(t),V_f^+(t))$ constructed in Lemma \ref{Lemma:IsometryPairHom} induces a homomorphism $\mathrm{Ad}((V_f,V_f^+))$ from $B^{p,\ast}_L(X)$ to $B^{p,\ast}_L(Y)\otimes M_2(\mathbb{C})$ defined by:
$$\mathrm{Ad}((V_f,V_f^+))(u)(t)=V_f(t)(u(t)\oplus 0)V_f^+(t)$$
for any $u\in B^{p,\ast}_L(X)$ and $t\in [0,\infty)$, such that
  $$\mathrm{prop}(\mathrm{Ad}((V_f,V_f^+))(u)(t))\leq \sup_{(x,y)\in \mathrm{supp}(u(t))}\{d(f(x),f(y))\}+2\varepsilon_k + 2\varepsilon_{k+1}$$
for all $t\in [k,k+1]$.\par
Moreover, the map $\mathrm{Ad}((V_f,V_f^+))_*$ induced by $\mathrm{Ad}((V_f,V_f^+))$ on $K$-theory depends only on f and not on the choice of the pairs $(V_k,V_k^+)$ in the construction of $V_f(t)$ and $V_f^+(t)$.
\end{lemma}
\begin{proof}
    The proof of this lemma is similar to the proof of Lemma \ref{Lemma:IsometryPairHom}.
\end{proof}

\begin{definition}
    The $i$-th \textit{dual $L^p$ $K$-homology} is defined to be $K_i(B^{p,\ast}_L(X))$.
\end{definition}

\subsection{Strongly Lipschitz homotopy invariance of (dual) $L^p$ $K$-homology}

    In this section, we will prove that (dual) $L^p$ $K$-homology is strongly Lipschitz homotopy invariant. In the following, we just discuss the case of dual $L^p$ $K$-homology; similarly, we can obtain the same result for $L^p$ $K$-homology.
\begin{lemma}\label{Lemma:LipschitzHomotopy}
    Let $f$ and $g$ be two Lipschitz maps from $X$ to $Y$, let $F(t,x)$ be a strongly Lipschitz homotopy connecting $f$ and $g$, then
              $$\mathrm{Ad}((V_f,V^+_f))_*=\mathrm{Ad}((V_g,V^+_g))_*: K_*(B^{p,\ast}_L(X)) \rightarrow K_*(B^{p,\ast}_L(Y))$$
\end{lemma}
\begin{proof}
    We just prove this lemma for $K_1$ group, and by suspension, we can obtain the same result for $K_0$ group.
    Choose $\{t_{i,j}\}_{i\geq 0,j\geq 0} \subseteq [0,1]$ satisfying \par
(1) $t_{0,j}=0, t_{i,j+1}\leq t_{i,j}, t_{i+1,j}\geq t_{i,j}$;\par
(2) there exists $N_j\rightarrow \infty$ such that $t_{i,j}=1$ for all $i\geq N_j$ and $N_{j+1}\geq N_j$; \par
(3) $d(F(t_{i+1,j},x),F(t_{i,j},x))<\varepsilon_j=1/(j+1), d(F(t_{i,j+1},x),F(t_{i,j},x))<\varepsilon_j$ for all $x\in X$.\par
    Let $f_{i,j}(x)=F(t_{i,j},x)$, by Lemma \ref{Lemma:CoveringIsometry}, there exist an isometric operator $V_{f_{i,j}}:E_X^p\rightarrow E_Y^p$ and a contractive operator $V^+_{f_{i,j}}:E_Y^p\rightarrow E_X^p$ with $V^+_{f_{i,j}}V_{f_{i,j}}=I$ such that
              $$\mathrm{supp}(V_{f_{i,j}}) \subseteq \{(x,y)\in X\times Y:d(f_{i,j}(x),y)<1/(1+i+j)\};$$
              $$\mathrm{supp}(V^+_{f_{i,j}}) \subseteq\{(y,x)\in Y\times X:d(f_{i,j}(x),y)<1/(1+i+j)\}.$$
    For each $i>0$, define a family of operators $V_i(t)(t\in [0,\infty))$ from $E^p_X\oplus E^p_X$ to $E^p_Y\oplus E^p_Y$ and a family of operators $V^+_i(t)(t\in [0,\infty))$ from $E^p_Y\oplus E^p_Y$ to $E^p_X\oplus E^p_X$ by
              $$V_i(t)=R(t-j)(V_{f_{i,j}}\oplus V_{f_{i,j+1}})R^*(t-j), t\in[j,j+1]$$
              $$V^+_i(t)=R(t-j)(V^+_{f_{i,j}}\oplus V^+_{f_{i,j+1}})R^*(t-j), t\in[j,j+1]$$
where
                  $$R(t)=\begin{pmatrix}
                          \cos(\pi t/2)  & \sin(\pi t/2)\\
                          -\sin(\pi t/2) & \cos(\pi t/2)
                         \end{pmatrix}.$$
For any $[u]\in K_1(B^{p,\ast}_L(X))$, consider:
    $$u_0(t)=\mathrm{Ad}((V_f,V^+_f))(u)=V_f(t)(u(t)\oplus I)V^+_f(t)+(I-V_f(t)V^+_f(t));$$
    $$u_{\infty}(t)=\mathrm{Ad}((V_g,V^+_g))(u)=V_g(t)(u(t)\oplus I)V^+_g(t)+(I-V_g(t)V^+_g(t));$$
    $$u_i(t)=\mathrm{Ad}((V_i,V^+_i))(u)=V_i(t)(u(t)\oplus I)V^+_i(t)+(I-V_i(t)V^+_i(t)),$$
For each $i$, define $n_i$ by
                       $$n_i=\left\{
                               \begin{array}{lr}
                               \max\{j:i\geq N_j\}, &  \{j:i\geq N_j\}\not= \emptyset;\\
                               0,                  &    \{j:i\geq N_j\}= \emptyset.
                               \end{array}
                               \right. $$
We can choose $V_{f_{i,j}}$ in such a way that: $u_i(t)=u_{\infty}$ where $t\leq n_i$. \par
Define
                   $$w_i(t)=u_i(t)(u^{-1}_{\infty}(t))$$
Consider
                  \begin{align*}
                   a&=\bigoplus^{\infty}_{i=0}(w_i\oplus I);\\
                   b&=\bigoplus^{\infty}_{i=0}(w_{i+1}\oplus I);\\
                   c&=(I\oplus I)\bigoplus^{\infty}_{i=1}(w_i\oplus I).
                  \end{align*}
By the construction of $\{t_{i,j}\}$, we know that $a,b,c\in (B^{p,\ast}_L(X) \otimes M_2(\mathbb{C}))^+$. It is not difficult to see that $a$ is equivalent to $b$ and $b$ is equivalent to $c$ in $K_1(B^{p,\ast}_L(X))$. Thus $u_0u^{-1}_{\infty} \oplus_{i\geq 1} I$ is equivalent to $\oplus_{i\geq 0} I$ in $K_1(B^{p,\ast}_L(X))$. This means that $\mathrm{Ad}((V_f,V^+_f))_*=\mathrm{Ad}((V_g,V^+_g))_*$.
\end{proof}

\begin{corollary}
If $X$ is strongly Lipschitz homotopy equivalent to $Y$, then they have the same (dual) $L^p$ K-homology.
\end{corollary}

\subsection{Cutting and pasting of the (dual) $L^p$ $K$-homology}

    Let $X$ be a simplicial complex endowed with the $\ell^1$-metric, and let $X_1$ be a simplicial subcomplex of $X$. For $p\in (1,\infty)$ define $B^{p,\ast}_L(X_1;X)$ to be the closed subalgebra of $B^{p,\ast}_L(X)$ generated by all elements $f$ such that there exists $c_t>0$ satisfying $\lim _{t\rightarrow \infty} c_t=0$ and $\mathrm{supp}(f(t)) \subset \{(x,y)\in X \times X: d((x,y),X_1 \times X_1)\leq c_t\}$ for all $t\in [0,\infty)$.
\begin{lemma}\label{Lemma:LocalizationIndepStar}
    The inclusion homomorphism $i$ from $B^{p,\ast}_L(X_1)$ to $B^{p,\ast}_L(X_1;X)$ induces an isomorphism from $K_*(B^{p,\ast}_L(X_1))$ to $K_*(B^{p,\ast}_L(X_1;X))$.
\end{lemma}
\begin{proof}
    For any $\varepsilon>0$, let $B_{\varepsilon}(X_1)=\{x \in X : d(x,X_1) \leq \varepsilon\}$. There exists a small $\varepsilon_0>0$ such that $B_{\varepsilon_0}(X_1)$ is strongly Lipschitz homotopy equivalent to $X_1$. Any element in $K_1(B^{p,\ast}_L(X_1;X))$ can be represented by an invertible element $a\in (B^{p,\ast}_L(X_1;X))^+$ such that $a=a'+I$ and there exists $c_t>0$ satisfying $\lim_{t\rightarrow \infty} c_t=0$, $\mathrm{supp}(a'(t))\subset \{(x,y)\in X \times X: d((x,y),X_1 \times X_1)\leq c_t\}$. Uniform continuity of $a(t)$ implies that $a(t+st_0)(s\in [0,1])$ is norm continuous in $s$ for all $t_0>0$. Thus $[a(t)]$ is equivalent to $[a(t+st_0)]$ in $K_1(B^{p,\ast}_L(X_1;X))$ for any $t_0$. We can choose $t_0$ large enough so that $\mathrm{supp}(a'(t+t_0)) \subset B_{\varepsilon_0}(X_1) \times B_{\varepsilon_0}(X_1)$ for all $t$. By Lemma \ref{Lemma:LipschitzHomotopy}, we know that $i_*$ is surjective. \par
    A similar argument can be used to show that $i_*$ is injective. The case for $K_0$ can be similarly dealt with by a suspension argument.
\end{proof}

\begin{lemma}
    Let $X$ be a simplicial complex endowed with the $\ell^1$-metric, let $X_1,X_2$ be its two simplicial subcomplexes. We have the following six term exact sequence:
    \begin{footnotesize}
    \[
    \begin{CD}
    K_0(B^{p,\ast}_L(X_1 \cap X_2))   @>>>   K_0(B^{p,\ast}_L(X_1))\oplus K_0(B^{p,\ast}_L(X_2))   @>>>    K_0(B^{p,\ast}_L(X_1 \cup X_2))\\
             @AAA                                   & &                                                   @VVV     \\
    K_1(B^{p,\ast}_L(X_1 \cup X_2))   @<<<   K_1(B^{p,\ast}_L(X_1))\oplus K_0(B^{p,\ast}_L(X_2))   @<<<    K_1(B^{p,\ast}_L(X_1 \cap X_2))
    \end{CD}
    \]
    \end{footnotesize}
\end{lemma}
\begin{proof}
    Let $Y=X_1 \cup X_2$, observe that $B^{p,\ast}_L(X_1;Y)$ and $B^{p,\ast}_L(X_2;Y)$ are ideals of $B^{p,\ast}_L(Y)$ such that $B^{p,\ast}_L(X_1;Y)+B^{p,\ast}_L(X_2;Y)=B^{p,\ast}_L(Y)$. And $B^{p,\ast}_L(X_1;Y) \cap B^{p,\ast}_L(X_2;Y)=B^{p,\ast}_L(X_1 \cap X_2;Y)$ since $(X_1,X_2)$ is the strong excision pair of $Y$ (Definition \ref{Def:ControlledBoundary}). Then by the Mayer-Vietoris sequence for $K$-theory of Banach algebras and Lemma \ref{Lemma:LocalizationIndepStar}, we can obtain this lemma.
\end{proof}

\begin{remark}
    By similar argument as above, we have the following six term exact sequence for $L^p$ localization algebra:
\begin{small}
    \[
    \begin{CD}
    K_0(B^p_L(X_1 \cap X_2))   @>>>   K_0(B^p_L(X_1))\oplus K_0(B^p_L(X_2))   @>>>    K_0(B^p_L(X_1 \cup X_2))\\
             @AAA                                   & &                                     @VVV     \\
    K_1(B^p_L(X_1 \cup X_2))   @<<<   K_1(B^p_L(X_1))\oplus K_0(B^p_L(X_2))   @<<<    K_1(B^p_L(X_1 \cap X_2))
    \end{CD}
    \]
\end{small}
\end{remark}

\subsection{Main result and proof}
    Let $X$ be a finite dimensional simplicial complex endowed with $\ell^1$-metric. Recall that $E^p_X$ is the $L^p$-$X$-module and $\mathcal{B}^*(E^p_X)$ is the Banach algebra of all dual-operators on $E_X^p$ for $p\in (1,\infty)$.  Every $T\in \mathcal{B}^*(E^p_X)$ can be viewed as an element in $B(E^p_X)$. Thus the inclusion map induces a contractive homomorphism
                           $$\phi:B^{p,\ast}_L(X)\rightarrow B^p_L(X)$$ \par

    Next we use the Riesz-Thorin interpolation theorem to build a connection between the dual $L^p$ localization algebra and the localization $C^*$-algebra. Firstly, let us recall this interpolation theorem.
\begin{lemma} [Riesz-Thorin]
    Let $(X,\mu)$ and $(Y,\nu)$ be two measure spaces. Let $T$ be a linear operator defined on the set of all simple functions on $X$ and taking values in the set of measurable functions on $Y$. Let $1 \leq p_0,p_1,q_0,q_1 \leq \infty$ and assume that
                          $$||T(f)||_{L^{q_0}}\leq M_0||f||_{L^{p_0}},$$
                          $$||T(f)||_{L^{q_1}}\leq M_1||f||_{L^{p_1}},$$
for all simple functions $f$ on $X$. Then for all $0<\theta<1$ we have
                          $$||T(f)||_{L^{q'}}\leq M^{1-\theta}_0M^{\theta}_1||f||_{L^{p'}}$$
for all simple functions $f$ on $X$, where $1/p'=(1-\theta)/p_0+\theta/p_1$ and $1/q'=(1-\theta)/q_0+\theta/q_1$.\\
By density, $T$ has a unique extension as a bounded operator from $L^{p'}(X,\mu)$ to $L^{q'}(Y,\nu)$.
\end{lemma}

    For any $p\in (1,\infty)$, let $q$ be the dual number of $p$, i.e. $1/p+1/q=1$. Let $p_0=q_0=p$, $p_1=q_1=q$ and $\theta=1/2$ in the above. By the Riesz-Thorin interpolation theorem, we have that each element $T\in\mathcal{B}^*(E^p_X)$ can be considered as an element  in $B(E^2_X)$. This correspondence induces a contractive homomorphism
                           $$\psi:B^{p,\ast}_L(X) \rightarrow C^*_L(X),$$
where $C^*_L(X)$ is the localization $C^*$-algebra of $X$.

\begin{proposition}\label{Thm:pIndependentLocalization}
    Let $X$ be a finite dimensional simplicial complex endowed with $\ell^1$-metric, then for any $p\in (1,\infty)$, $\psi$ induces an isomorphism between $K_*(B^{p,\ast}_L(X))$ and $K_*(C^*_L(X))$.
\end{proposition}
\begin{proof}
    Let $X^{(n)}$ be the $n$-skeleton of $X$, we shall prove this theorem for $X^{(n)}$ by induction on $n$. \par
    When $n=0$, $K_*(B^{p,\ast}_L(X^{(0)}))$ equals to the direct product of $K_*(\mathcal{K}^*(\ell^p))$ and $K_*(C^*_L(X^{(0)}))$ equals to the direct product of $K_*(\mathcal{K}^*(l^2))$ using the fact that the algebra of all bounded and uniformly continuous functions from $[0,\infty)$ to a Banach algebra has the same $K$-theory as this Banach algebra. Then by Corollary \ref{Cor:KtheoryLpCompact}, $\psi_*$ is an isomorphism in this case. \par
    Assume by induction that the theorem holds when $n=m-1$. Next we shall prove the theorem holds when $n=m$. For each simplex $\triangle$ of dimension $m$ in $X$, we let
    $$\triangle_1=\{x\in \triangle: d(x,c(\triangle))\leq 1/100\},\triangle_2=\{x\in \triangle: d(x,c(\triangle))\geq 1/100\},$$
where $c(\triangle)$ is the center of $\triangle$. \par
    Let
      $$X_1=\bigcup_{\triangle:\text{simplex of dimension $m$ in X}} \triangle_1;$$
      $$X_2=\bigcup_{\triangle:\text{simplex of dimension $m$ in X}} \triangle_2.$$
Notice that:
\begin{enumerate}
\item $X_1$ is strongly Lipschitz homotopy equivalent to
                  $$\{c(\triangle): \text{$\triangle$ is $m$-dimensional simplex in X}\};$$
\item $X_2$ is strongly Lipschitz homotopy equivalent to $X^{(m-1)}$;\\
\item $X^{(m)}=X_1 \cup X_2$ and $X_1 \cap X_2$ is the disjoint union of the boundaries of all $m$-dimensional $\triangle_1$ in $X^{(m)}$.
\end{enumerate}
    (1) and (2) together with the strongly Lipschitz homotopy invariance of the dual $L^p$-$K$-homology and the induction hypothesis, imply that the theorem holds for $X_1$ and $X_2$.\par
    By the strongly Lipschitz homotopy invariance of $K_*(B^{p,\ast}_L(X))$, $K_*(C^*_L(X))$ and the cutting and pasting exact sequence, we also know that our lemma holds for $X_1\cap X_2$. \par
    Thus we can complete our induction process by using the cutting and pasting exact sequence and the five lemma.
\end{proof}

    Using a similar argument for $\phi$, we have the following theorem.
\begin{proposition}\label{Thm:pIndenpendentStarLocalization}
    Let $X$ be a finite dimensional simplicial complex endowed with the $\ell^1$-metric, then for any $p\in (1,\infty)$, $\phi$ induces an isomorphism between $K_*(B^{p,\ast}_L(X))$ and $K_*(B^p_L(X))$.
\end{proposition}

    By Propositions \ref{Thm:pIndependentLocalization} and \ref{Thm:pIndenpendentStarLocalization}, we obtain that the $K$-theory for $L^p$ localization algebra is independent of $p$ for a finite dimensional simplicial complex. This give a partial answer to the question 26 in \cite{ChungNowakPCBC} proposed by Chung and Nowak.
\begin{proposition}
    Let $X$ be a finite dimensional simplicial complex endowed with the $\ell^1$-metric, $K_*(B^p_L(X))$ does not depend on $p\in (1,\infty)$.
\end{proposition}

    Further more, we have the following $p$-indenpendency of $K$-theory for $L^p$ Roe algebra.
\begin{corollary}
    Let $X$ be a proper metric space, assume that there exists an anti-\v Cech system $\{C_k\}_k$ for $X$ such that $N_{C_k}$ is a finite dimensional simplicial complex for all $k$. Then if for all $p\in (1,\infty)$, the $L^p$ coarse Baum-Connes conjecture is true for $X$, we have that $K_*(B^p(X))$ does not depend on $p$.
\end{corollary}

    By the Theorem \ref{Thm:LpBCCFAD}, we have the following theorem.
\begin{theorem}\label{Thm:MainThmPIndpendency}
    Let $X$ be a proper metric space. If $X$ has finite asymptotic dimension, then $K_*(B^p(X))$ does not depend on $p$ for $p\in (1,\infty)$.
\end{theorem}

\section{Open problems}
In this last section, we list several interesting open problems.
\begin{question}
 There are several versions of $L^p$ $K$-homology. Are they all the same?
\end{question}
There are many different versions of $K$-homology
        \begin{enumerate}
            \item Kasparov's $K$-homology \cite{KasparovKhomology}
            \item $K$-theory of dual algebra by Paschke \cite{PaschkeDuality}
            \item $K$-theory of localization algebra of Guoliang Yu \cite{YuLocalization}\cite{QiaoRoe}
            \item Localization $K$-homology by Xiaoman Chen and Qin Wang \cite{ChenWangLocalization}
            \item $E$-theory by Connes and Higson  \cite{ConnesHigson}
        \end{enumerate}
        In the $L^2$ case, all above concepts are the same.
        In this paper, we have seen that the $L^p$ counterpart of all above notions may be equivalent for finite dimensional simplicial complex. But we are not very optimistic that they are equivalent for general topological spaces. To prove the equivalence, we need some deep theorems, say the Voiculescu Theorem \cite{Voiculescu} and the Kasparov Technical Lemma \cite{Kasparov88}, for $L^p$ spaces.
     \begin{question}
     Is it possible to prove that the $K$-theory of $L^p$ Roe algebras are independent of $p$ without using the coarse Baum-Connes conjecture?
     \end{question}
     Up to now, all the results about the $p$ indenpendence of the $K$-theory of the group algebras, crossed products and Roe algebras rely on the Baum-Connes conjecture or the coarse Baum-Connes conjecture, since the $K$-homology sides are easier to maneuver. A more direct approach without using the (coarse) Baum-Connes conjecture would shed some light on a Banach algebra approach to the (coarse) Baum-Connes conjecture. For example, if we know certain groups admitting proper isometric actions on $L^p$-spaces and the $K$-theory of their $L^p$ group algebras does not depend on $p$, by the result of Kasparov and Yu \cite{KasparovYuPBC}, we can verify the Baum-Connes conjecture for these groups.
     \begin{question}
     Can we develop an $L^p$ version of Dirac-dual-Dirac method for the $L^p$ Baum-Connes conjecture for amenable groupoids?
    \end{question}
    In \cite{TuATmenable}, Tu showed the Baum-Connes conjecture is true for amenable groupoids, or more generally, a-T-menable groupoids. For a space with finite asymptotic dimension, or more generally finite decomposition complexity, the coarse groupoid is amenable. For a dynamical system with finite dynamical complexity, the corresponding transformation groupoid is also amenable \cite{GuentnerWillettYuFDC}. It would be great if we can modify Tu's method to deal with $L^p$ groupoid algebras and give a unified proof for Chung's result on $L^p$ crossed products and the results in this paper on $L^p$ Roe algebras.
 \begin{question}
    Does the $L^p$ coarse Baum-Connes conjecture hold for spaces coarsely embeddable into Hilbert spaces?
 \end{question}
    Recently, Shan and Wang verified the $L^p$ coarse Novikov conjecture for spaces coarsely embeddable into simply connected nonpositively curved manifolds \cite{ShanWang}. The key ingredient in the proof is the $L^p$ version of Yu's twisted Roe algebra technique \cite{YuEmbedding}. Shan and Wang's theorem is the first positive result on the injectivity of the $L^p$ coarse Baum-Connes conjecture using Yu's technique. Can we also use the $L^p$ version of Yu's technqiue to give a surjective argument when the space is coarsely embeddable into a Hilbert space? More generally, can we prove the $L^p$ coarse Baum-Connes conjecture for spaces coarsely embeddable into $L^p$ spaces?
 \begin{question}
    What will happen if we use $L^p(X,\nu)$ or a general $L^p$-space $E$ as $L^p$-$X$-module to define $L^p$ Roe algebra and $L^p$ localization algebra?
 \end{question}
    The proof of Lemma \ref{Lemma:CoveringIsometry} does not work for this broader definition. Thus we need to find a new way to construct the homomorphism between $L^p$ Roe algebras covering the map between the underlying spaces.
 \begin{question}
 Are there any topological and geometric implications of the $L^p$ (coarse) Baum-Connes conjecture?
 \end{question}
 For example, does it imply the Gromov conjecture \cite{GromovConjecture} that uniformly contractible manifolds with bounded geometry admit no uniform positive scalar curvature?
 \begin{question}
 Are there any counter-examples for the injectivity of the $L^p$ coarse Baum-Connes conjecture?
 \end{question}
 In \cite{HigsonLafforgueSkandalis}\cite{WillettYuExpander}, Higson-Lafforgue-Skandalis and Willett-Yu showed that Magulis type expanders and expanders with large girth are counter-examples for the surjectivity of the coarse Baum-Connes conjecture. In \cite{ChungNowakPCBC},  Chung and Nowak showed that Margulis type expanders are still counter-example for the $L^p$ coarse Baum Connes conjecture. However, the existence of injectivity counterexample of the $L^p$ coarse Baum Connes conjecture is still open.
 In \cite{YuFAD}, Guoliang Yu gave a counterexample of the injectivity of coarse Baum-Connes conjecture. The proof relies on a positive scalar curvature argument. Is Yu's counter-example still a counter-example for the $L^p$ version?

\bibliographystyle{plain}
\bibliography{LpCBC}

\end{document}